\documentclass[a4paper]{article}
\usepackage{amssymb,amsmath,amsfonts,amsthm,graphicx,verbatim}
\theoremstyle{plain}
\newtheorem{proposition}{Proposition}[section]
\newtheorem{theorem}[proposition]{Theorem}
\newtheorem{corollary}[proposition]{Corollary}

\newtheorem{lemma}[proposition]{Lemma}
\newtheorem{conjecture}[proposition]{Conjecture}

\theoremstyle{definition}
\newtheorem{algorithm}[proposition]{Algorithm}

\newcommand{\Split}{\mathsf{Split}}
\newcommand{\Merge}{\mathsf{Merge}}
\newcommand{\Reduce}{\mathsf{Reduce}}

\newcommand{\Tri}{\mathbf{Tri}}
\newcommand{\DTri}{\mathbf{DTri}}
\newcommand{\abg}{\alpha,\beta,\gamma}

\newcommand{\wa}{\hat{a}}
\newcommand{\wb}{\hat{b}}
\newcommand{\wc}{\hat{c}}

\newcommand{\Gi}   {G_1}
\newcommand{\Gii}  {G_2}
\newcommand{\Giii} {G_3}
\newcommand{\Giv}  {G_4}
\newcommand{\Gv}   {G_5}
\newcommand{\Gvi}  {G_6}
\newcommand{\Gvii} {G_7}
\newcommand{\Gviii}{G_8}
\newcommand{\Gix}  {G_9}
\newcommand{\Gx}   {G_{10}}
\newcommand{\Gxi}  {G_{11}}
\newcommand{\Gxii} {G_{12}}
\newcommand{\Gxiii}{G_{13}}
\newcommand{\Gxiv} {G_{14}}
\newcommand{\Gbar} {\overline{\Gxiv}}

\begin{document}
\title{The minimal density of triangles in tripartite graphs}
\author{Rahil Baber\thanks{Department of Mathematics, UCL, London, WC1E 6BT, UK. Email: rahilbaber@hotmail.com.} \and J. Robert Johnson\thanks{School of Mathematical Sciences, Queen Mary University of London, E1 4NS, UK. Email: r.johnson@qmul.ac.uk.}\and John Talbot\thanks{Department of Mathematics, UCL, London, WC1E 6BT, UK. 
Email: talbot@math.ucl.ac.uk.  This author is a Royal Society University Research Fellow.}}
\date\today
\maketitle
\begin{abstract}
We determine the minimal density of triangles in a
tripartite graph with prescribed edge densities. This extends a
previous result of Bondy, Shen, Thomass\'e and Thomassen
characterizing those edge densities guaranteeing the existence of a
triangle in a tripartite graph.

To be precise we show that a suitably weighted copy of the graph
formed by deleting a certain $9$-cycle from $K_{3,3,3}$ has minimal
triangle density among all weighted tripartite graphs with
prescribed edge densities.
\end{abstract}

\section{Introduction}

Extremal questions for triangles in graphs have a very long history.
The first such result, Mantel's theorem \cite{Man}, tells us that a
graph with $n$ vertices and more than $n^2/4$ edges must contain at
least one triangle.

For graphs with more than $n^2/4$ edges it is natural to pose a
quantitative question: what is the minimum number of triangles in a
graph with a given number of edges? In this direction Razborov
\cite{Raz} determined (asymptotically) the minimal density of
triangles in a graph of given edge density. This recent result was
the cumulation of decades of contributions on this question due to
Bollob\'as \cite{Boll}, Erd\H os \cite{Erd62}, Lov\'asz and Simonovits \cite{LS}, and
Fisher \cite{Fisher}.

Recently Bondy, Shen, Thomass\'e and Thomassen \cite{Paper_BSTT}
considered the very natural question of when a tripartite graph with
prescribed edge densities must contain a triangle. (A tripartite
graph is a graph $G=(V,E)$ for which there exists a partition of its
vertices into three vertex classes such that all edges go between
classes. The edge density between a pair of vertex classes $X,Y$ is
simply the proportion of edges present between the two classes:
$|E(X,Y)|/|X||Y|$.)

Bondy et al.~characterized those triples of edge densities
guaranteeing a triangle in a tripartite graph. As a special case
they showed that any tripartite graph in which the density of edges
between each pair of classes is greater than $1/\varphi=0.618\ldots$
contains a triangle (a precise statement of their full result can be
found in the next section).

The aim of this paper is to prove a quantitative result which
extends the theorem of Bondy et al. in the same way that Razborov's
result extends Mantel's theorem.

The remainder of the paper is organised as follows. Formal
definitions and main results are given in the next section. Our main
result splits into two rather different cases and the following two
sections contain their proofs. We finish with some conjectures and
open problems. The proof relies on a computer search and an appendix
containing the C++ source code for this computation is also
included.

\section{Definitions and results}

A tripartite graph is a graph $G=(V,E)$ for which there exists a
partition of its vertices into three independent sets. Throughout,
whenever we consider a tripartite graph we will implicitly assume
that a fixed tripartition $V=A\dot{\cup}B \dot{\cup}C$ is given.

A \emph{weighted tripartite graph} $(G,w)$ is a tripartite graph
$G=(V,E)$ together with a \emph{weighting} $w:V\to [0,1]$ satisfying
\[
\sum_{a\in A}w(a) = \sum_{b\in B}w(b) = \sum_{c\in C} w(c)=1.
\]
The weight of an edge $xy\in E(G)$ is  $w(xy)=w(x)w(y)$. The
\emph{edge densities} of $(G,w)$ are
\begin{align*}
\alpha(G,w)=\!\!\!\!\sum_{bc\in E(B,C)}\!\!\!w(bc),\quad
\beta(G,w)=\!\!\!\!\sum_{ac\in E(A,C)}\!\!\!w(ac),\quad
\gamma(G,w)=\!\!\!\!\sum_{ab\in E(A,B)}\!\!\!w(ab).
\end{align*}
We denote the set of all weighted tripartite graphs by $\Tri$. For
$\abg\in [0,1]$ we let $\Tri(\abg)$ denote the set of all weighted
tripartite graphs with edge densities $\alpha(G,w)=\alpha$, $
\beta(G,w)=\beta$, $\gamma(G,w)=\gamma$.

Let $(G,w)\in\Tri$. A \emph{triangle} in $G$ is a set of three
vertices, $a\in A,b\in B, c\in C$, such that $ab,ac,bc\in E(G)$. We
denote the set of all triangles in $G$ by $T(G)$.  The weight of a
triangle $xyz\in T(G)$ is $w(xyz)=w(x)w(y)w(z)$. The \emph{triangle
density} of $(G,w)\in \Tri$ is \[t(G,w)=\sum_{abc\in T(G)}w(abc).\]

Note that with the obvious definitions of edge and triangle
densities for simple tripartite graphs any such graph can be
converted into a weighted tripartite graph with the same edge and
triangle densities by setting the vertex weights to be $1/|A|,
1/|B|, 1/|C|$ for vertices in classes $A,B,C$ respectively.

Also, any weighted tripartite graph with rational weights can be
converted into a simple tripartite graph with the same edge and
triangle densities by taking a suitable blow-up. To be precise,
choose an integer $n$ so that $nw(v)$ is an integer for all vertices
$v$ and replace each vertex of weight $x$ with $nx$ new vertices.
The new vertices being clones of the old in the sense that we join a
pair of vertices in the new graph if and only if the pair of
vertices they arise from are adjacent in the weighted graph.

We are interested in how small the triangle density of a weighted
tripartite graph with prescribed edge densities can be. Formally we
wish to determine the following function. For $\abg\in [0,1]$ let
\[
T_{\min}(\abg)=\min_{(G,w)\in \Tri(\abg)}t(G,w).
\]

It is not difficult to believe that this function is well-defined
however for completeness we sketch a proof of this fact. Since this
makes use of results from much later in the paper we suggest the
reader takes this on trust until they reach the relevant results.
Given $0\leq \abg\leq 1$, Lemma \ref{DoublyWeighted_SameAs_Weighted}
implies that $\Tri(\abg)\neq \emptyset$. Now Lemma
\ref{MaxClassSize3} implies that when attempting to minimise
$t(G,w)$ over $\Tri(\abg)$ we may restrict our search to the finite
subfamily consisting of tripartite graphs with at most three
vertices per class. Finally note that for a single tripartite graph
$G$ the problem of determining the minimum value of $t(G,w)$,
subject to the edge densities of $(G,w)$ being $\abg$, is a
minimisation problem for a continuous function over a compact
domain. Hence $T_{\min}(\abg)$ is well-defined.

The following simple lemma shows that solving this weighted problem
will give an asymptotic answer to the question of how many triangles
a simple (unweighted) tripartite graph with given edge densities
must have.
\begin{lemma}\label{SimpleG}
\begin{enumerate}
\item[(i)] If $G$ is a simple tripartite graph with edge densities $\abg$
then it has triangle density at least $T_{\min}(\abg)$.
\item[(ii)] For rational $\abg$, if $(H,w)\in\Tri(\abg)$ then for all
$\epsilon>0$ there is a simple tripartite graph $G$ with edge
densities $\abg$ and triangle density at most $t(H,w)+\epsilon$.
\end{enumerate}
\end{lemma}

\begin{proof}
Part (i) is immediate since any tripartite graph can be transformed
into a weighted tripartite graph by weighting vertices uniformly in
each vertex class as described above.

For part (ii) let $w'$ be a rational weighting of $H$ such that if
the edge densities of $(H,w')$ are $\alpha',\beta',\gamma'$ we have
$|\alpha-\alpha'|,|\beta-\beta'|,|\gamma-\gamma'|,|t(H,w)-t(H,w')|<\frac{\epsilon}{4}$.
We can do this since for a given $H$ the edge and triangle densities
are continuous functions of the vertex weights. Now choose an
integer $n$ so that $nw'(v)$ is an integer for all vertices $v$, and
$n^2|\alpha-\alpha'|,n^2|\beta-\beta'|,n^2|\gamma-\gamma'|$ are all
integers. Blow up $H$ by replacing each vertex $v$ with $nw'(v)$
cloned vertices to form a simple graph $G'$ with $n$ vertices in
each class. Finally add or remove at most $\frac{3\epsilon}{4}n^2$
edges from $G'$ to form a graph $G$ with edge densities $\abg$. This
creates at most $\frac{3\epsilon}{4}n^3$ new triangles and so the
triangle density of $G$ is at most
$t(H,w')+\frac{3\epsilon}{4}<t(H,w)+\epsilon$.
\end{proof}

Bondy, Shen, Thomass\'e and Thomassen \cite{Paper_BSTT} proved the
following sharp Tur\'an-type result. If $(G,w)\in \Tri(\abg)$ and
$(\abg)\in R$, where
\begin{align*}
R = \{(\abg)\in[0,1]^3 : \alpha\beta+\gamma>1, \alpha\gamma+\beta>1,
\beta\gamma+\alpha>1\},
\end{align*}
then $G$ must contain a triangle.

\begin{theorem}\label{Thm_T=0}
$T_{\min}(\abg) = 0 \iff (\abg)\in [0,1]^3\setminus R$.
\end{theorem}
In particular, $T_{\min}(d,d,d)=0$ if and only if $d\leq0.618\ldots$
(the positive root of the quadratic $x^2+x-1=0$).

Our main result (Theorem \ref{Thm_H9}) determines the minimal
density of triangles in a weighted tripartite graph with prescribed
edge densities.

The \emph{tripartite complement} of a tripartite graph $G$ is the
graph obtained by deleting the edges of $G$ from the complete
tripartite graph on the same vertex classes as $G$. Let $H_9$ be the
graph whose tripartite complement is given in Figure \ref{Fig_H9}.
\begin{theorem}\label{Thm_H9}
For any $(\abg)\in R$ there exists a weighting $w$ of $H_9$ such
that $(H_9,w)\in\Tri(\abg)$ and $t(H_9,w)=T_{\min}(\abg)$.
\end{theorem}
\begin{figure}[tbp]
\begin{center}
\includegraphics[height=4cm]{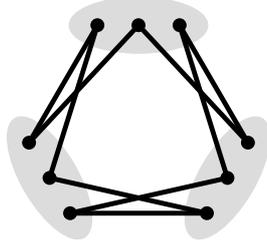}
\caption{The tripartite complement of the graph
$H_9$.}\label{Fig_H9}
\end{center}
\end{figure}
This theorem combined with Lemma \ref{SimpleG} shows that a suitable
blow-up of $H_9$ has asymptotically the minimum density of triangles
for given edge densities.

There are two distinct cases to consider in the proof of Theorem
\ref{Thm_H9}, depending on the values of $\abg$. Let
\[
\Delta(\abg)= \alpha^2+\beta^2+\gamma^2
-2\alpha\beta-2\alpha\gamma-2\beta\gamma +4\alpha\beta\gamma.
\]
We partition $R$ into two regions: $R_1$ and $R_2$ where
\[R_1=\{(\abg)\in R:\Delta(\abg)\geq 0\}\] and $R_2=
R\setminus R_1$. For $R_1$ we have the following result.
\begin{theorem}\label{Thm_H6}
If $(\abg)\in R_1$ and $H_6$ is the graph whose tripartite
complement is given in Figure \ref{Fig_H6} then there exists a
weighting $w$ such that $(H_6,w)\in\Tri(\abg)$, and for any such $w$
\[
T_{\min}(\abg)=t(H_6,w)=\alpha+\beta+\gamma-2.\]
\end{theorem}

\begin{figure}[tbp]
\begin{center}
\includegraphics[height=4cm]{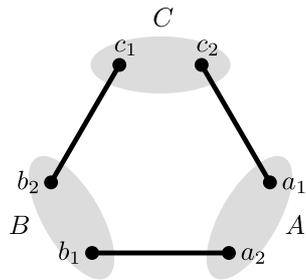}
\caption{The tripartite complement of the graph
$H_6$.}\label{Fig_H6}
\end{center}
\end{figure}


Let $(G,w)\in\Tri(\abg)$. If $t(G,w)=T_{\min}(\abg)$ then $(G,w)$ is
said to be \emph{extremal}. If there does not exist
$(G',w')\in\Tri(\abg)$ with $t(G',w')=t(G,w)$ and $|V(G')|<|V(G)|$
then $(G,w)$ is said to be \emph{vertex minimal}. The tripartite
graphs $G$ and $H$ with specified tripartitions are
\emph{strongly-isomorphic} if there is a graph isomorphism $f:G\to
H$ such that the image of each vertex class in $G$ is a vertex class
in $H$.

\begin{figure}[tbp]
\begin{center}
\includegraphics[height=4cm]{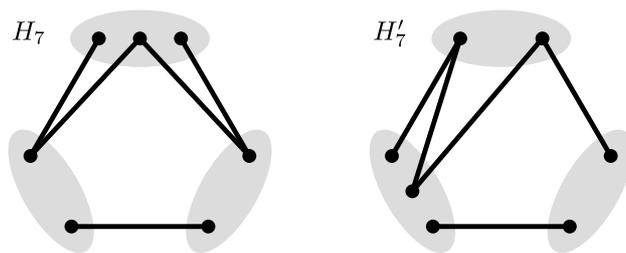}
\caption{The tripartite complements of the graphs $H_7$ and
$H_7'$.}\label{Fig_H7H7'}
\end{center}
\end{figure}

\begin{theorem}\label{Thm_H7H7'H9}
If $(\abg)\in R_2$ and $(G,w)\in\Tri(\abg)$ is extremal and vertex
minimal then $G$ is strongly-isomorphic to $H_7, H_7',$ or $H_9$
(see Figure \ref{Fig_H7H7'}).
\end{theorem}

\begin{proof}[Proof of Theorem \ref{Thm_H9}]
The graphs $H_6$, $H_7$ and $H_7'$ are induced subgraphs of $H_9$
hence Theorems \ref{Thm_H6} and \ref{Thm_H7H7'H9} imply Theorem
\ref{Thm_H9}.
\end{proof}

We conjecture that in fact the extremal graph is always an
appropriate weighting of $H_7$. This would also give a simple
formula for $T_{\min}(\abg)$. See section \ref{Sec_Conjectures} for
details.

\section{Proof of Theorem \ref{Thm_H6} (the region $R_1$)}

\begin{lemma}\label{a+b+g-2_LowerBound}
For any $\abg\in [0,1]$ and $(G,w)\in\Tri(\abg)$ we have
\begin{align*}
t(G,w)\geq\alpha+\beta+\gamma-2.
\end{align*}
\end{lemma}

\begin{proof}
Define
\begin{align*}
\mathbf{1}_{xy} =
\begin{cases}
1, &\text{if $xy\in E(G)$},\\
0, &\text{otherwise},
\end{cases}\qquad\qquad\mathbf{1}_{xyz} =
\begin{cases}
1, &\text{if $xyz\in T(G)$},\\
0, &\text{otherwise}.
\end{cases}
\end{align*}
Given $abc\in A\times B\times C$, the number of edges present
between these three vertices is at most two unless $abc$ forms a
triangle. Hence
\begin{align}\label{Eq_SumIndicator}
\sum_{abc\in A\times B\times
C}w(abc)(\mathbf{1}_{bc}+\mathbf{1}_{ac}+\mathbf{1}_{ab})\leq\sum_{abc\in
A\times B\times C}w(abc)(2+\mathbf{1}_{abc}).
\end{align}
The LHS of (\ref{Eq_SumIndicator}) sums to $\alpha+\beta+\gamma$,
while the RHS is $2+t(G,w)$. Therefore
$t(G,w)\geq\alpha+\beta+\gamma-2$.
\end{proof}

\begin{lemma}\label{6Vertex_IsBest}
If $w$ is a weighting of $H_6$ satisfying $(H_6,w)\in\Tri(\abg)$
then
\begin{align*}
t(H_6,w)=\alpha+\beta+\gamma-2=T_{\min}(\abg).
\end{align*}
\end{lemma}

For ease of notation the weight associated with a vertex is
indicated with a hat above the label, for example $w(b_1)$ is
represented as $\wb_1$.

\begin{proof}
Consider a general weighting of $H_6$ with vertices as labelled in
Figure \ref{Fig_H6}. We know $\wa_2=1-\wa_1, \wb_2=1-\wb_1$ and
$\wc_2=1-\wc_1$ since the sum of the weights of the vertices in a
class add up to one. Hence we can express the densities in terms of
only $\wa_1, \wb_1,$ and $\wc_1$. The edge densities of $H_6$ are
\begin{align*}
\alpha = 1-\wc_1+\wb_1\wc_1,\quad \beta = 1-\wa_1+\wa_1\wc_1,\quad
\gamma = 1-\wb_1+\wa_1\wb_1.
\end{align*}
The triangle density is given by
\begin{align*}
t(H_6,w) &= \wa_1\wb_1\wc_1+(1-\wa_1)(1-\wb_1)(1-\wc_1)\\
&=1-\wa_1-\wb_1-\wc_1+\wa_1\wb_1+\wa_1\wc_1+\wb_1\wc_1\\
&=\alpha+\beta+\gamma-2.
\end{align*}
By Lemma \ref{a+b+g-2_LowerBound} we have $t(H_6,w)=T_{\min}(\abg)$.
\end{proof}

We now need to determine for which $(\abg)\in R$ a weighting $w$
exists such that $(H_6,w)\in\Tri(\abg)$

\begin{lemma}\label{abgInR3_implies_Not01}
\begin{enumerate}
\item[(i)] If  $(\abg)\in R$ then $\abg>0$.
\item[(ii)] If $(\abg)\in R_2$ then $0<\abg<1$.
\end{enumerate}
\end{lemma}

\begin{proof}
If $(\abg)\in R$ and $\alpha=0$ then, since $\alpha\beta+\gamma>1$,
we have $\gamma>1$, a contradiction. Similarly  $\beta,\gamma> 0$.

If $(\abg)\in R_2$ then $R_2\subseteq R$ implies that $\abg>0$. If
$\alpha=1$ then
$\Delta(\abg)=\Delta(1,\beta,\gamma)=(1-\beta-\gamma)^2\geq 0$. But
$(\abg)\in R_2$ implies that $\Delta(\abg)<0$, a contradiction.
Similarly $\beta,\gamma< 1$.
\end{proof}

\begin{lemma}\label{6Vertex_ValidRegion}
For $(\abg)\in R$ there exists a weighting $w$ of $H_6$ such that
$(H_6,w)\in\Tri(\abg)$ if and only if $(\abg)\in R_1$.
\end{lemma}

\begin{proof}
If $(\abg)\in R$ then Lemma \ref{abgInR3_implies_Not01} (i) implies
that $\abg\neq 0$. First we will prove that if $(\abg)\in R$ and
there exists a weighting $w$ such that $(H_6,w)\in \Tri(\abg)$, then
$(\abg)\in R_1$.

Let us label the vertices of $H_6$ as in Figure \ref{Fig_H6}.
Suppose $w$ is weighting of $H_6$ such that $(H_6,w)\in\Tri(\abg)$.
The edge densities in terms of $\wa_1, \wb_1, \wc_1$ are
\begin{align}
\alpha &= 1-\wc_1+\wb_1\wc_1,\label{Eq_6V_Alpha}\\
\beta &= 1-\wa_1+\wa_1\wc_1,\label{Eq_6V_Beta}\\
\gamma &= 1-\wb_1+\wa_1\wb_1.\label{Eq_6V_Gamma}
\end{align}

\textit{Case 1: One of $\abg$ equals one.}\\
By Lemma \ref{abgInR3_implies_Not01} (ii) $(\abg)\notin R_2$. Hence
$(\abg)\in R$ implies $(\abg)\in R_1$.

\textit{Case 2: $\abg\neq 1$.}\\
Since $\abg\neq 1$ we have $\wa_1,\wb_1,\wc_1\neq 0,1$. Rearranging
(\ref{Eq_6V_Gamma}) and (\ref{Eq_6V_Beta}) we can write $\wb_1$ and
$\wc_1$ in terms of $\wa_1$
\begin{align}
\wb_1 &= \frac{1-\gamma}{1-\wa_1} \label{Eq_6V_b}\\
\wc_1 &= \frac{\wa_1+\beta-1}{\wa_1}. \label{Eq_6V_c}
\end{align}
Substituting into (\ref{Eq_6V_Alpha}) and simplifying gives
\begin{align*}
\alpha \wa_1^2+(-\alpha+\beta-\gamma)\wa_1+\gamma-\beta\gamma = 0.
\end{align*}
Hence
\begin{align}\label{Eq_6V_Quad_a}
\wa_1 &= \frac{\alpha-\beta+\gamma\pm\sqrt{\Delta(\abg)}}{2\alpha},
\intertext{substituting back into (\ref{Eq_6V_b}) and
(\ref{Eq_6V_c}) gives} \wb_1 &=
\frac{\alpha+\beta-\gamma\pm\sqrt{\Delta(\abg)}}{2\beta},
\label{Eq_6V_Quad_b}\\
\wc_1 &= \frac{-\alpha+\beta+\gamma\pm\sqrt{\Delta(\abg)}}{2\gamma}.
\label{Eq_6V_Quad_c}
\end{align}
By the definition of a weighting we have
$\wa_1,\wb_1,\wc_1\in\mathbb{R}$, hence $\Delta(\abg)\geq 0$, and so
$(\abg)\in R_1$.

Next we will show that if $(\abg)\in R_1$ then there exists a
weighting $w$ such that $(H_6,w)\in\Tri(\abg)$.

\textit{Case 1: One of $\abg$ equals one.}\\
Without loss of generality suppose $\alpha=1$. Since
$(1,\beta,\gamma)\in R_1\subseteq R$ we have $\beta+\gamma>1$. It is
easy to check that $\wa_1=\gamma, \wb_1=1,
\wc_1=(\beta+\gamma-1)/\gamma$ satisfy
(\ref{Eq_6V_Alpha}),(\ref{Eq_6V_Beta}),(\ref{Eq_6V_Gamma}) and
$\wa_1,\wb_1,\wc_1\in [0,1]$ when $\beta+\gamma>1$. This is enough
to define a weighting $w$ of $H_6$.

\textit{Case 2: $\abg\neq 1$.}\\
Since $\Delta(\abg)\geq 0$, we may define
$\wa_1,\wb_1,\wc_1\in\mathbb{R}$ by (\ref{Eq_6V_Quad_a}),
(\ref{Eq_6V_Quad_b}), (\ref{Eq_6V_Quad_c}), taking the positive
square root in each case. Due to the way $\wa_1,\wb_1,\wc_1$ were
constructed they satisfy (\ref{Eq_6V_Alpha}), (\ref{Eq_6V_Beta}),
(\ref{Eq_6V_Gamma}). Hence if $\wa_1,\wb_1,\wc_1$ form a weighting
$w$ we will have $(H_6,w)\in\Tri(\abg)$. We need only prove
$\wa_1,\wb_1,\wc_1\in(0,1)$.

We will prove $\wa_1\in(0,1)$, the proofs of $\wb_1,\wc_1\in(0,1)$
follow similarly. If $0<\alpha-\beta+\gamma$ then $0<\wa_1$ because
$\wa_1$ is the positive square root version of (\ref{Eq_6V_Quad_a}).
Now $(\abg)\in R$ implies
$0<\alpha\beta+\gamma-1<\alpha+\gamma-\beta$, and consequently
$0<\wa_1$. By (\ref{Eq_6V_Quad_a}) if
$\sqrt{\Delta(\abg)}<\alpha+\beta-\gamma$ then $\wa_1<1$. Again
$(\abg)\in R$ implies that $0<\alpha\gamma + \beta -1<\alpha+\beta
-\gamma$. Hence if we can show
$\Delta(\abg)<(\alpha+\beta-\gamma)^2$ we will be done. Expanding
and simplifying yields $0<4\alpha\beta(1-\gamma)$ which is true
because $\abg\in (0,1)$.
\end{proof}

\begin{proof}[Proof of Theorem \ref{Thm_H6}]
The result follows immediately from Lemma \ref{6Vertex_IsBest} and
\ref{6Vertex_ValidRegion}.
\end{proof}

\section{Proof of Theorem \ref{Thm_H7H7'H9} (the region $R_2$)}

We will begin by introducing a new type of graph in section
\ref{Sec_Properties} which will allow us to develop a series of
conditions that extremal vertex minimal examples must satisfy. In
section \ref{Sec_Search} we outline an algorithm that allows us to
utilize the results of section \ref{Sec_Properties} to search for
the extremal vertex minimal graphs in a finite time. This algorithm
produces fourteen possible graphs. In section
\ref{Sec_SpecificGraphs} we eliminate those not strongly-isomorphic
to $H_7, H_7'$ and $H_9$ by analysing each of them in turn.

\subsection{Properties}\label{Sec_Properties}
Our proof strategy for Theorem \ref{Thm_H7H7'H9} is to establish
various properties any extremal and vertex minimal weighted
tripartite graph must satisfy. To prove these properties we
introduce a new type of tripartite graph.

A \emph{doubly-weighted tripartite graph} $(G,w,p)$ is a weighted
tripartite graph $(G,w)\in\Tri$ together with a function $p: E(G)\to
(0,1]$. We denote the set of all doubly-weighted tripartite graphs
by $\DTri$. If $(G,w,p)\in\DTri$ then the \emph{weight} of an edge
$xy\in E(G)$ is defined to be
\begin{align*}
\lambda(xy) = w(xy)p(xy).
\end{align*}
The \emph{edge density} between a pair of vertex classes $X$ and $Y$
is
\begin{align*}
\sum_{xy\in E(X,Y)}\lambda(xy).
\end{align*}
The \emph{triangle density} is defined as
\begin{align*}
t(G,w,p) = \sum_{abc\in T(G)}p(ab)p(ac)p(bc)w(abc).
\end{align*}

Any $(G,w)\in\Tri$ may be converted into a doubly-weighted
tripartite graph $(G,w,p)$ with the same triangle and edge densities
by adding the function $p:E(G)\to (0,1]$, $p(e)=1$ for all $e\in
E(G)$. Our next result allows us to do the reverse and convert a
doubly-weighted tripartite graph into a weighted tripartite graph,
leaving triangle and edge densities unchanged.

\begin{lemma}\label{DoublyWeighted_SameAs_Weighted}
Given $(G,w,p)\in\DTri$ there exists $(G',w')\in\Tri$ with the same
triangle and edge densities.
\end{lemma}

For $(G,w,p)\in\DTri$ we will say that $e\in E(G)$ is a
\emph{partial edge} if $p(e)<1$. To prove Lemma
\ref{DoublyWeighted_SameAs_Weighted} we need a process to eliminate
partial edges without affecting any of the densities.

For a graph $G$ and vertex $v\in V(G)$ let $\Gamma^G(v)$ denote the
neighbourhood of $v$ in $G$. When no confusion can arise we write
this simply as $\Gamma(v)$. Given a tripartite graph $G$ with a
vertex class $X$ and $v\in V(G)$ we write
$\Gamma_X^G(v)=\Gamma^G(v)\cap X$. Again when no confusion can arise
we write this simply as $\Gamma_X(v)$.
\begin{figure}[tbp]
\begin{center}
\includegraphics[height=4cm]{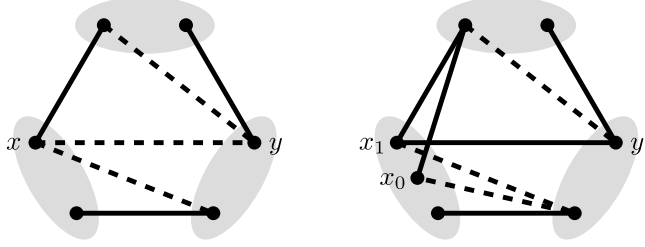}
\caption{An example of $(G,w,p)$ and $\Split(G,w,p,x,y)$. Partial
edges are represented by dotted lines, the solid lines are edges
which $p$ maps to $1$.}\label{Fig_Split}
\end{center}
\end{figure}
\begin{algorithm}[$\Split$]
The algorithm $\Split$ takes as input $(G,w,p)\in\DTri$ and an
ordered pair of vertices $(x,y)$, such that $xy$ is a partial edge.
Its output, $\Split(G,w,p,x,y)$, is a doubly-weighted tripartite
graph, which no longer contains the partial edge $xy$. If
$(G',w',p')=\Split(G,w,p,x,y)$ then $G',w',p'$ are formed as
follows:
\begin{itemize}
\item Construct $G'$ from $G$ by replacing the vertex $x$ by two new
vertices $x_0$ and $x_1$ that lie in the same vertex class as $x$. Add
edges from $x_0, x_1$ so that $\Gamma^{G'}(x_0)=\Gamma^{G}(x)\setminus \{y\}$
and $\Gamma^{G'}(x_1)=\Gamma^{G}(x)$.
\item Set $w'(x_0)=w(x)(1-p(xy))$ and $w'(x_1)=w(x)p(xy)$.
Let $w'(v)=w(v)$ for all $v\in V(G)\setminus\{x\}$.
\item Set $p'(x_0v)=p'(x_1v)=p(xv)$ for all $v\in \Gamma^{G}(x)\setminus\{y\}$,
and $p'(x_1y)=1$. Let $p'(uv)=p(uv)$ for all $uv\in E(G)$ such that
$u,v\neq x$.
\end{itemize}
Note that in $\Split(G,w,p,y,x)$ (the result of applying $\Split$ to
$(G,w,p)$ and $(y,x)$) the vertex $y$  would have been ``split''
into two new vertices rather than $x$. It also does not contain the
partial edge $xy$. So if we wish to remove the partial edge $xy$ we
can choose between $\Split(G,w,p,x,y)$ and $\Split(G,w,p,y,x)$.
\end{algorithm}
Figure \ref{Fig_Split} shows an example application of $\Split$ with
``before'' and ``after'' pictures of $(G,w,p)$ and
$\Split(G,w,p,x,y)$.

\begin{lemma}\label{SplitDensity}
For any $(G,w,p)\in\DTri$ and $xy$ a partial edge,
$(G',w',p')=\Split(G,w,p,x,y)$ has the same triangle and edge
densities as $(G,w,p)$.
\end{lemma}

\begin{proof}
Without loss of generality let us assume $x\in A$ and $y\in B$. We
will prove the result by calculating the difference in densities
between $(G',w',p')$ and $(G,w,p)$ and showing them to be zero. The
change in the edge density between classes $A$ and $B$ is
\begin{multline*}
w(y)(w'(x_1)p'(x_1y)-w(x)p(xy))+\\
\sum_{v\in \Gamma_B^G(x)\setminus\{y\}}
w(v)(w'(x_0)p'(x_0v)+w'(x_1)p'(x_1v)-w(x)p(xv))
\end{multline*}
which is zero. Similarly the change in density between classes $A$
and $C$ is zero. There is no change in the density between classes
$B$ and $C$ since the algorithm $\Split$ leaves this part of the
graph untouched. The change in the triangle density is
\begin{multline*}
\sum_{xuv\in T(G), u\in B\setminus\{y\}, v\in
C}\Bigl(w'(x_0)+w'(x_1)-w(x)\Bigr)
w(u)w(v)p(xu)p(xv)p(uv)+\\
\sum_{xyv\in T(G), v\in C}\Bigl(w'(x_1)p'(x_1y)-w(x)p(xy)\Bigr)
w(y)w(v)p(xv)p(yv),
\end{multline*}
which is zero, hence the triangle and edge densities do not change.
\end{proof}

\begin{proof}[Proof of Lemma \ref{DoublyWeighted_SameAs_Weighted}]
Given $(G,w,p)\in\DTri$, if $p(e)=1$ for all $e\in E(G)$ then the
weighted tripartite graph $(G,w)$ will have the same densities as
the doubly-weighted tripartite graph.

Suppose $(G,w,p)$ contains a partial edge $av$, with $a\in A$. We
can remove this partial edge by replacing $(G,w,p)$ by
$\Split(G,w,p,a,v)$. Unfortunately this may introduce new partial
edges. However, we can show that by repeated applications of
$\Split$ we will eventually remove all partial edges. Consider
\[
Z(G,w,p) = \sum_{v\in A}3^{d_z(v)},
\]
where
\[
d_z(v) = |\{u\in V(G) : uv\in E(G), p(uv)\neq 1\}|.
\]
If $(G',w',p')=\Split(G,w,p,a,v)$ then $Z(G',w',p')<Z(G,w,p)$. This
is because $\Split$ replaces vertex $a$ with the vertices $a_0$ and
$a_1$, and so $Z$ changes by
\begin{align*}
3^{d_z(a_0)}+3^{d_z(a_1)}-3^{d_z(a)} &=
3^{d_z(a)-1}+3^{d_z(a)-1}-3^{d_z(a)}\\
&= -3^{d_z(a)-1}.
\end{align*}
Since $Z$ is integral and is bounded below (by zero for instance),
repeatedly applying $\Split$ will eventually remove all partial
edges incident with $A$. Note that doing this will not have created
any new partial edges between classes $B$ and $C$.

We can repeat this process on the partial edges leaving $B$, to get
rid of the remaining partial edges. Let us call the resulting
doubly-weighted tripartite graph $(G'',w'',p'')$. Since we created
$(G'',w'',p'')$ only by applying $\Split$, by Lemma
\ref{SplitDensity}, $(G,w,p)$ and $(G'',w'',p'')$ must have the same
edge and triangle densities. Since $(G'',w'',p'')$ has no partial
edges, $p''(e)=1$ for all $e\in E(G'')$, consequently $(G'',w'')$
has the same edge and triangle densities as $(G'',w'',p'')$ and
therefore $(G,w,p)$.
\end{proof}

Since we can convert easily between weighted and doubly-weighted
tripartite graphs, it is useful to know when there exist
doubly-weighted tripartite graphs with the same edge densities but
with smaller triangle densities. Let $(G,w,p)$ be a doubly-weighted
tripartite graph. By carefully modifying $p$ we can adjust the
weights of edges whilst not affecting the edge densities and
potentially decreasing the triangle density. Our next result lists a
series of conditions under which this could occur.

Let $G$ be a tripartite graph with vertex classes $A,B,C$. For $a\in
A, b\in B$ define
\[
C_{ab}= \{c\in C : ac,bc\in E(G)\}.
\]
\begin{lemma}\label{Order_Sset}
If $(G,w,p)\in\DTri$ satisfies conditions $(i)-(iv)$, given below,
then there exists $(G',w,p')\in\DTri$ with the same edge densities
as $(G,w,p)$ but $t(G',w,p')<t(G,w,p)$.
\begin{itemize}
\item[(i)] $w(v)>0$ for all $v\in V(G)$,
\item[(ii)] $p(e)=1$ for all $e\in E(A,C)\cup E(B,C)$,
\item[(iii)] there exist, not necessarily distinct, vertices $a_0,a_1\in A$, $b_0,b_1\in B$ such that $a_1b_1\in E(G)$ and
either $a_0b_0\notin E(G)$ or $p(a_0b_0)< 1$,
\item[(iv)] $\displaystyle\sum_{c\in C_{a_0b_0}} w(c)<\displaystyle\sum_{c\in C_{a_1b_1}} w(c)$.
\end{itemize}
\end{lemma}

\begin{corollary}\label{Order_subset}
Let $(G,w)\in\Tri$. If there exist, not necessarily distinct,
vertices $a_0,a_1\in A$, $b_0,b_1\in B$ such that $a_0b_0\notin
E(G)$, $a_1b_1\in E(G)$ and $C_{a_0b_0}$ is a proper subset of
$C_{a_1b_1}$ then $(G,w)$ is either not extremal or not vertex
minimal.
\end{corollary}

\begin{proof}[Proof of Corollary \ref{Order_subset}]
We will prove that if $(G,w)$ is vertex minimal then it is not
extremal by applying Lemma \ref{Order_Sset}.

Let $(G,w)$ be vertex minimal, so $w(v)> 0$ for all $v\in V(G)$. We
can add the function $p$ which maps all edges of $G$ to $1$ to
create $(G,w,p)\in\DTri$. Now $(G,w,p)$ has the same triangle and
edge densities as $(G,w)$. By Lemma
\ref{DoublyWeighted_SameAs_Weighted} it is enough to show that there
exists $(G',w',p')\in\DTri$ with the same edge densities as
$(G,w,p)$ but a smaller density of triangles. Note that conditions
$(i)-(iii)$ in the statement of Lemma \ref{Order_Sset} hold for
$(G,w,p)$. Thus Lemma \ref{Order_Sset} will provide such a
$(G',w',p')$ if we can show that
\[
\sum_{c\in C_{a_0b_0}} w(c)<\sum_{c\in C_{a_1b_1}} w(c).
\]
Let $u\in C_{a_1b_1}\setminus C_{a_0b_0}$. Since $(G,w)$ is vertex
minimal $w(u)>0$. Hence
\begin{align*}
\sum_{c\in C_{a_1b_1}} w(c) - \sum_{c\in C_{a_0b_0}} w(c) =
\sum_{c\in C_{a_1b_1}\setminus C_{a_0b_0}} w(c) \geq w(u)>0.
\end{align*}
In which case all the conditions of Lemma \ref{Order_Sset} are
satisfied, and $(G,w)$ is not extremal.
\end{proof}

\begin{proof}[Proof of Lemma \ref{Order_Sset}]
If $a_0b_0\notin E(G)$ let $G'$ be the graph produced from $G$ by
adding the edge $a_0b_0$. If $a_0b_0\in E(G)$ then let $G'=G$.
Define $p':E(G')\to (0,1]$ by $p'(e)=p(e)$ for $e\in E(G')\setminus
\{a_0b_0, a_1b_1\}$ and
\begin{align*}
p'(a_0b_0) &=
\begin{cases}
\dfrac{\delta}{w(a_0)w(b_0)}, &\text{if $a_0b_0\notin E(G)$,}\\
p(a_0b_0)+\dfrac{\delta}{w(a_0)w(b_0)}, &\text{if $a_0b_0\in E(G)$,}
\end{cases}\\
p'(a_1b_1) &= p(a_1b_1)-\frac{\delta}{w(a_1)w(b_1)},
\end{align*}
where $\delta>0$ is chosen sufficiently small so that
$p'(a_0b_0),p'(a_1b_1)\in(0,1)$.

The weights and edges have not changed between classes $A,C$ and
$B,C$. Consequently the corresponding edge densities will be the
same in $(G,w,p)$ and $(G',w,p')$. However, the edge density between
class $A$ and $B$, and the triangle densities may have changed. The
difference in the $A,B$ edge density between $(G',w,p')$ and
$(G,w,p)$ is
\begin{align*}
w(a_0)w(b_0)\frac{\delta}{w(a_0)w(b_0)}
-w(a_1)w(b_1)\frac{\delta}{w(a_1)w(b_1)}=0.
\end{align*} The change in triangle density is
\begin{align*}
\sum_{a_0b_0c\in T(G')}w(a_0)w(b_0)w(c)\frac{\delta}{w(a_0)w(b_0)}
-\sum_{a_1b_1c\in T(G')}w(a_1)w(b_1)w(c)\frac{\delta}{w(a_1)w(b_1)}
\end{align*}
which simplifies to
\begin{align*}
\delta\left(\sum_{a_0b_0c\in T(G')}w(c)-\sum_{a_1b_1c\in
T(G')}w(c)\right)<0.
\end{align*}
Where the final inequality follows from condition $(iv)$.

Hence the density of triangles in $(G',w,p')$ is smaller than that
in $(G,w,p)$, but the edge densities are the same in both.
\end{proof}

\begin{figure}[tbp]
\begin{center}
\includegraphics[height=4cm]{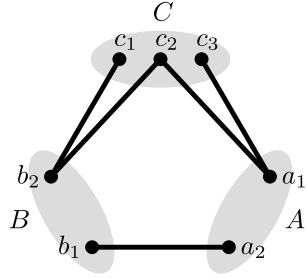}
\caption{The tripartite complement of the graph
$H_7$.}\label{Fig_H7}
\end{center}
\end{figure}

\begin{lemma}\label{7Vertex_MinDensity}
Consider the graph $H_7$ whose tripartite complement is given in
Figure \ref{Fig_H7}. If $(\abg)\in R_2$ then there exists a
weighting $w$ such that $(H_7,w)\in\Tri(\abg)$ and
$t(H_7,w)=2\sqrt{\alpha\beta(1-\gamma)}+2\gamma-2$. Furthermore
$t(H_7,w)\leq t(H_7,w')$ for all weightings $w'$, such that
$(H_7,w')\in\Tri(\abg)$.
\end{lemma}

\begin{proof}
If $(\abg)\in R_2$ then, by Lemma \ref{abgInR3_implies_Not01} (ii),
we know that $0<\abg<1$. Consider a general weighting of $H_7$, with
vertices labelled as in Figure \ref{Fig_H7}. If such a weighting of
$H_7$ has edge densities $\abg$ then $\abg<1$ implies that
$\wa_1\neq 0,1,\gamma$. Now given $\abg$ and $\wa_1\neq 0,1,\gamma$
we have enough information to deduce the rest of the weights of the
vertices. (Note that this may not be an actual weighting since some
of these values may lie outside of $[0,1]$.)
\begin{align*}
\wa_2 &= 1-\wa_1, &\wb_1 &= \frac{1-\gamma}{1-\wa_1}, &\wb_2 &=\frac{\gamma-\wa_1}{1-\wa_1},\\
\wc_1 &= 1-\frac{1-\beta}{\wa_1}, &\wc_3 &= 1-\frac{(1-\alpha)(1-\wa_1)}{\gamma-\wa_1},\\
\wc_2 &=
\frac{1-\beta}{\wa_1}+\frac{(1-\alpha)(1-\wa_1)}{\gamma-\wa_1}-1,
\end{align*}
which have been deduced from
\begin{align*}
 1 &= \wa_1+\wa_2, &1-\gamma &= \wa_2\wb_1, &1 &= \wb_1+\wb_2,\\
1-\beta &= (1-\wc_1)\wa_1, &1-\alpha &= (1-\wc_3)\wb_2,  &1
&=\wc_1+\wc_2+\wc_3,
\end{align*}
respectively. There are two triangles in $H_7$, with weights
$\wa_1\wb_1\wc_1$ and $\wa_2\wb_2\wc_3$, hence the triangle density
is
\begin{align*}
\wa_1\left(\frac{1-\gamma}{1-\wa_1}\right)\left(1-\frac{1-\beta}{\wa_1}\right)
+(1-\wa_1)\left(\frac{\gamma-\wa_1}{1-\wa_1}\right)
\left(1-\frac{(1-\alpha)(1-\wa_1)}{\gamma-\wa_1}\right)
\end{align*}
which simplifies to
\begin{align*}
2\gamma-2+\frac{\beta(1-\gamma)}{1-\wa_1}+\alpha(1-\wa_1).
\end{align*}
This expression is minimized when
$1-\wa_1=\sqrt{\beta(1-\gamma)/\alpha}$, and consequently we obtain
the desired triangle density of
$2\sqrt{\alpha\beta(1-\gamma)}+2\gamma-2$. We now must show that the
vertex weights implied by $\wa_1=1-\sqrt{\beta(1-\gamma)/\alpha}$
all lie in $[0,1]$ and that $\wa_1\neq\gamma,0,1$. Since the sum of
the weights in each class equals one, in order to show that all
weights lie in $[0,1]$ it is sufficient to show that they are all
non-negative.

If $\wa_1=\gamma$ then $1-\gamma = \sqrt{\beta(1-\gamma)/\alpha}$,
which rearranges to $\alpha\gamma+\beta-\alpha=0$. However,
$\alpha\gamma+\beta-\alpha>\alpha\gamma+\beta-1>0$ (as $(\abg)\in
R_2\subseteq R$), hence $\wa_1\neq \gamma$. $1-\wa_1$ is clearly
positive, proving that $0<\wa_2$ and $\wa_1\neq 1$. Showing
$0<\wa_1$ is equivalent to proving $\sqrt{\beta(1-\gamma)/\alpha}<1$
which is true if $0<\beta\gamma+\alpha-\beta$, and this holds
because $\beta\gamma+\alpha-\beta>\beta\gamma+\alpha-1>0$. Since
$\wb_2$ equals $1-\sqrt{\alpha(1-\gamma)/\beta}$, a similar argument
shows that $\wb_1,\wb_2>0$. It is also straightforward to show that
$\wc_1,\wc_3 >0$, but showing $\wc_2>0$ requires more work.

Using $\wc_1+\wc_2+\wc_3=1$, $\wc_1 = 1-(1-\beta)/\wa_1$, and
$\wc_3= 1-(1-\alpha)/\wb_2$ we obtain
\begin{align*}
\wc_2&= -1+\frac{(1-\beta)\wb_2 + (1-\alpha)\wa_1}{\wa_1\wb_2}.
\end{align*}
Hence $\wc_2>0$ if and only if
\[
\wa_1\wb_2<(1-\beta)\wb_2 + (1-\alpha)\wa_1.
\]
Substituting $\wa_1 =1-\sqrt{\beta(1-\gamma)/\alpha}$ and $\wb_2 =
1-\sqrt{\alpha(1-\gamma)/\beta}$
yields
\begin{align*}
\alpha+\beta-\gamma<2\sqrt{\alpha\beta(1-\gamma)}.
\end{align*}
Now $\alpha+\beta-\gamma>\alpha\gamma+\beta-1>0$, hence $0<\wc_2$ if
and only if $(\alpha+\beta-\gamma)^2<4\alpha\beta(1-\gamma)$.
Collecting all the terms onto the left hand side shows that we
require $\Delta(\abg)<0$, which we have from the fact that
$(\abg)\in R_2$.
\end{proof}

\begin{lemma}\label{CyclicIneq_NotExtremal}
For any $(\abg)\in R_2$,
\begin{align*}
T_{\min}(\abg)<\min\{\alpha\beta+\gamma-1, \alpha\gamma+\beta-1,
\beta\gamma+\alpha-1\}.
\end{align*}
\end{lemma}

\begin{proof}[Proof of Lemma \ref{CyclicIneq_NotExtremal}]
Without loss of generality let us assume that
\[
\alpha\beta+\gamma-1 = \min\{\alpha\beta+\gamma-1,
\alpha\gamma+\beta-1, \beta\gamma+\alpha-1\}.\] By Lemma
\ref{7Vertex_MinDensity} we know that for any $(\abg)\in R_2$ there
exists a weighting $w$ such that $(H_7,w)\in\Tri(\abg)$ and
$t(H_7,w) = 2\sqrt{\alpha\beta(1-\gamma)}+2\gamma-2$. Hence
$T_{\min}(\abg)\leq 2\sqrt{\alpha\beta(1-\gamma)}+2\gamma-2$. If
$\alpha\beta+\gamma-1\leq 2\sqrt{\alpha\beta(1-\gamma)}+2\gamma-2$
then
\[
\alpha\beta+1-\gamma \leq 2\sqrt{\alpha\beta(1-\gamma)}.
\]
Squaring and rearranging yields
\[
(\alpha\beta+\gamma-1)^2\leq 0.
\]
Since $(\abg)\in R_2\subseteq R$ we know $\alpha\beta+\gamma-1>0$
holds true, hence we have a contradiction.
\end{proof}

\begin{lemma}\label{ClassSize1_NotExtremal}
Let $(\abg)\in R_2$. If $(G,w)\in\Tri(\abg)$ is extremal, then $|A|,
|B|, |C|\geq 2$.
\end{lemma}
To prove Lemma \ref{ClassSize1_NotExtremal} we will require the
following algorithm.
\begin{figure}[tbp]
\begin{center}
\includegraphics[height=4cm]{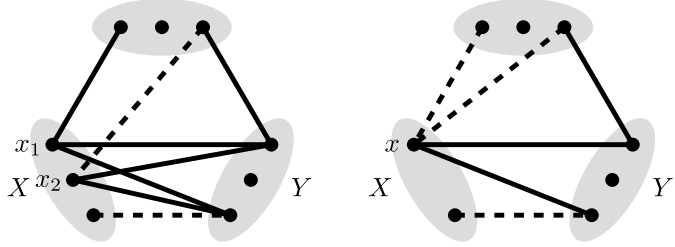}
\caption{An example of $(G,w,p)$ and $\Merge(G,w,p,x_1,x_2)$.
Partial edges are represented by dotted lines, the solid lines are
edges which $p$ maps to $1$}\label{Fig_Merge}
\end{center}
\end{figure}
\begin{algorithm}[$\Merge$]
The algorithm $\Merge$ takes as input $(G,w,p)\in\DTri$, and two
distinct vertices $x_1,x_2\in X$, where $X$ is one of the vertex
classes of $G$. The vertices $x_1, x_2$, must satisfy, for some
vertex class $Y\neq X$, $\Gamma_Y(x_1)=\Gamma_Y(x_2)$,
$w(x_1)+w(x_2)>0$ and $p(x_1y)=p(x_2y)=1$ for all
$y\in\Gamma_Y(x_1)$. The output of the algorithm is represented by
$\Merge(G,w,p,x_1,x_2)$ and is a doubly-weighted tripartite graph in
which $x_1, x_2$ have been replaced by a single new vertex: $x$. For
convenience let us write $(G',w',p')=\Merge(G,w,p,x_1,x_2)$. Now
$G', w', p'$ are formed as follows:
\begin{itemize}
\item Construct $G'$ from $G$ by replacing the vertices $x_1,x_2$ by
a new vertex $x$ in $X$. Add edges from $x$ so that
$\Gamma^{G'}(x)=\Gamma^G(x_1)\cup\Gamma^G(x_2)$.
\item Set $w'(x)=w(x_1)+w(x_2)$. Let $w'(v)=w(v)$ for all $v\in V(G')\setminus\{x\}$.
\item For $u,v\in V(G')\setminus\{x\}$ and $uv\in E(G')$, let $p'(uv)=p(uv)$.
For $xv\in E(G')$ set
\end{itemize}
\begin{align*}
p'(xv) =
\begin{cases}
w(x_1)p(x_1v)/w'(x), &\text{if $x_1v\in E(G)$, $x_2v\notin E(G)$},\\
w(x_2)p(x_2v)/w'(x), &\text{if $x_1v\notin E(G)$, $x_2v\in E(G)$},\\
(w(x_1)p(x_1v)+w(x_2)p(x_2v))/w'(x), &\text{if $x_1v\in E(G)$,
$x_2v\in E(G)$}.
\end{cases}
\end{align*}
Observe that for $y\in Y$ we have $xy\in E(G')$ if and only if
$x_1y,x_2y\in E(G)$ and in this case $p(xy)=1$. It is easy to check
that the edge and triangle densities of $(G,w,p)$ and $(G',w',p')$
are the same.
\end{algorithm}
\begin{figure}[tbp]
\begin{center}
\includegraphics[height=4cm]{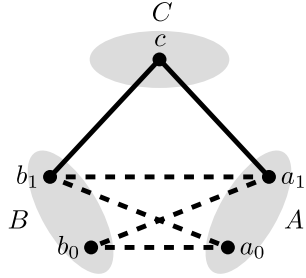}
\caption{A graph with $|C|=1$ after merging vertices in $A$ and $B$.
The dotted lines represent edges that may or may not be in the
graph.}\label{Fig_SingleVertexExample}
\end{center}
\end{figure}
\begin{proof}[Proof of Lemma \ref{ClassSize1_NotExtremal}]
Suppose $(G,w)$ is extremal and without loss of generality vertex
class $C=\{c\}$ contains exactly one vertex. We can assume $w(v)\neq 0$
for all $v\in V(G)$, as any vertices with weight zero can be removed
without affecting any of the densities. Create a doubly-weighted
tripartite graph $(G,w,p)$ with the same densities as $(G,w)$ by
setting $p(e)=1$ for all $e\in E(G)$. We will show that the triangle
density of $(G,w,p)$ is at least $\alpha\beta+\gamma-1$ and
consequently, by Lemma \ref{CyclicIneq_NotExtremal}, $(G,w)$ is not
extremal.

Since $(\abg)\in R_2$, by Lemma \ref{abgInR3_implies_Not01} (ii) we
have $\beta\neq 0,1$. Moreover since $C=\{c\}$ we know that there must exist a vertex in $A$
whose neighbourhood in $C$ is empty and another whose neighbourhood in $C$ is $\{c\}$. We can replace
all vertices $a\in A$ satisfying $\Gamma_C(a)=\emptyset$ by a single
vertex $a_0$ via repeated applications of the $\Merge$ algorithm on pairs
of such vertices. Similarly we can replace all vertices with
$\Gamma_C(a)=\{c\}$ by a single vertex $a_1$. Having done this we obtain a doubly-weighted graph in which $A=\{a_0,a_1\}$, $a_1c$ is
an edge, and $a_0c$ is a non-edge.  Note the edges and
weights between $B$ and $C$ remain unchanged but we may have
modified the edges and weights between $A$ and $B$.

By a similar argument we can reduce $B$ to two vertices $b_0,b_1$, with
$b_1c$ an edge and $b_0c$ a non-edge. Let us call this doubly
weighted graph $(G',w',p')$, and note it has the same densities as
$(G,w,p)$ and hence $(G,w)$. By construction we have
\[
a_0c,b_0c\notin E(G'),\quad a_1c,b_1c\in E(G'),\quad
p'(a_1c)=p'(b_1c)=1,
\]
see Figure \ref{Fig_SingleVertexExample}.

We now have enough information to determine the weights of all of
the vertices:
\begin{align*}
w'(c) = 1,\quad w'(a_1) = \beta,\quad w'(a_0) = 1-\beta,\quad
w'(b_1) = \alpha,\quad w'(b_0)=1-\alpha.
\end{align*}
The only information we are lacking about $(G',w',p')$ is which
edges are present in $E(A,B)$ and what their weights are. However,
since $(\abg)\in R$, Theorem \ref{Thm_T=0} implies that $G'$
contains a triangle. Hence $a_1b_1\in E(A,B)$. Since
$C_{a_1b_1}=\{c\}$ and $C_{a_0b_0}=C_{a_0b_1}=C_{a_1b_0}=\emptyset$,
Lemma \ref{Order_Sset} tells us that $(G,w)$ will not be extremal
unless $a_0b_0,a_0b_1,a_1b_0$ are all edges which $p'$ maps to $1$.

Now, $a_1b_1c$ is the only triangle in the doubly-weighted
tripartite graph, hence the triangle density is
$w'(a_1)w'(b_1)p'(a_1b_1)=\lambda(a_1b_1)$ (as $w'(c), p'(a_1c),
p'(b_1c)$ are all $1$). By the definition of edge density in a
doubly-weighted tripartite graph, we have
\begin{align*}
\gamma &=
\lambda(a_0b_0)+\lambda(a_0b_1)+\lambda(a_1b_0)+\lambda(a_1b_1)\\
&=(1-\alpha)(1-\beta)+\alpha(1-\beta)+(1-\alpha)\beta+t(G',w',p')\\
&=1-\alpha\beta+t(G',w',p')
\end{align*}
Hence the triangle density is $\alpha\beta+\gamma-1$, which by Lemma
\ref{CyclicIneq_NotExtremal} and Lemma
\ref{DoublyWeighted_SameAs_Weighted} implies that $(G,w)$ is not
extremal.
\end{proof}

\begin{lemma}\label{CollapseClass}
If $(\abg)\in R_2$,  $(G,w)\in\Tri(\abg)$ and for all $a_1,a_2\in
A$, $\Gamma_C(a_1)=\Gamma_C(a_2)$ then $(G,w)$ is not extremal.
\end{lemma}

\begin{proof}
If there exist any vertices with weight zero, we can remove them
without affecting the densities. Convert the resulting weighted
tripartite graph into a doubly-weighted tripartite graph and reduce
$A$ down to a single vertex, by repeated applications of $\Merge$ on
the vertices in $A$. Any partial edges that appear will lie in
$E(A,B)$.

Now repeatedly apply $\Split$ choosing to replace vertices in $B$
rather than $A$, until no more partial edges remain. Consequently we
have modified the weighted graph into a new weighted graph with the
same densities and now $|A|=1$. By Lemma
\ref{ClassSize1_NotExtremal} we know this is not extremal and hence
$(G,w)$ was not extremal.
\end{proof}


Our next lemma is an adaptation of a convexity argument by Bondy,
Shen, Thomass\'e and Thomassen (see proof of Theorem 3
\cite{Paper_BSTT}). This allows us to reduce the problem of
determining which tripartite graphs can be both vertex minimal and
extremal to those with at most three vertices in each vertex class.
\begin{lemma}\label{MaxClassSize3}
If $(G,w)\in\Tri$ is extremal and vertex minimal, then $|A|,
|B|,|C|\leq 3$
\end{lemma}
Again we introduce an algorithm to prove this lemma.
\begin{algorithm}[$\Reduce$]\label{Alg_Reduce}
The algorithm $\Reduce$ takes as input $(G,w)\in\Tri$ and a vertex
class $X$ of $G$, satisfying $|X|>3$. Its output, represented by
$\Reduce(G,w,X)$, is a weighted tripartite graph, which has the same
edge densities as $(G,w)$, but with $|X|\leq 3$, and triangle
density at most that of $(G,w)$.

To help explain the algorithm we will suppose $X=A$, (the other
choices of $X$ work similarly). For each vertex $a_i\in A$ let
\begin{align*}
\beta_i = \sum_{c\in\Gamma_C(a_i)}w(c),\quad \gamma_i =
\sum_{b\in\Gamma_B(a_i)}w(b),\quad t_i = \sum_{bc\in E(B,C),
a_ibc\in T(G)}w(bc).
\end{align*}
By definition
\begin{align*}
\beta = \sum_{i=1}^{|A|}w(a_i)\beta_i,\quad \gamma =
\sum_{i=1}^{|A|}w(a_i)\gamma_i,\quad t(G,w) =
\sum_{i=1}^{|A|}w(a_i)t_i.
\end{align*}
Consider the convex hull
\[
P=\left\{\sum_{i=1}^{|A|}x_i(\beta_i,\gamma_i,t_i) :
\sum_{i=1}^{|A|}x_i=1 \;\;\text{and}\;\; x_i\geq 0\right\}.
\]
Setting $x_i=w(a_i)$ shows that $(\beta,\gamma,t(G,w))$ lies in $P$.
By varying the values of the $x_i$ we can decrease the value of
$t(G,w)$ to $t'$ such that $(\beta,\gamma,t')$ lies on the boundary
of $P$. Moreover, by triangulating the facet of $P$ containing
$(\beta,\gamma,t')$, we can express $(\beta,\gamma,t')$ as a convex
combination of at most three elements of
$\{(\beta_i,\gamma_i,t_i):1\leq i\leq |A|\}$. Consequently we can
write
\[
(\beta,\gamma,t') = \sum_{i=1}^{|A|}x_i(\beta_i,\gamma_i,t_i)
\]
where $\sum x_i = 1$ and at most three of the $x_i$ are positive,
the rest are zero. Now define a new weighting $w'$ for $G$ by
$w'(a_i)=x_i$, $w'(v)=w(v)$ for $v\in V(G)\setminus A$. The weighted
tripartite graph $(G,w')$ has the same edge densities as $(G,w)$ and
a new triangle density $t'$ satisfying $t'\leq t(G,w)$. Furthermore
we can remove the zero weighted vertices from $A$ so that $|A|\leq
3$ and the densities are unchanged.
\end{algorithm}

\begin{proof}[Proof of Lemma \ref{MaxClassSize3}]
Suppose $(G,w)$ is extremal and vertex minimal with, without loss of
generality, $|A|>3$. Now, using Algorithm \ref{Alg_Reduce},
$\Reduce(G,w,A)$ has the same densities as $(G,w)$ (since $(G,w)$ is
extremal), but it has fewer vertices, contradicting the vertex
minimality of $(G,w)$.
\end{proof}


\begin{lemma}\label{OppositeClassSizeIs3}
Let $(G,w)$ be a weighted tripartite graph. If there exist distinct
vertices $a_1,a_2\in A$ with $\Gamma_C(a_1)=\Gamma_C(a_2)$ and
$|B|=3$, then $(G,w)$ is not extremal or not vertex minimal.
\end{lemma}

\begin{proof}
Convert $(G,w)$ into a doubly-weighted tripartite graph and replace
$a_1,a_2$ with a vertex $a$ by applying $\Merge$ (we may assume
$w(a_1)+w(a_2)>0$ by vertex minimality of $(G,w)$). Now $A$ has
reduced in size by one. If there are partial edges they will lie
between classes $A$ and $B$. Use the $\Split$ algorithm to remove
them, choosing to replace vertices in $B$ rather than $A$. Now
convert the doubly-weighted graph back into a weighted graph. This
weighted graph will have the same densities as $(G,w)$, $A$ has one
less vertex, and $|B|\geq 3$. If $|B|=3$ then this weighted graph is
of smaller order than $(G,w)$. If $|B|>3$ we can use $\Reduce$ to
modify the weights of vertices in $B$, such that at most three of
them have a non-zero weight. Simply remove all vertices with zero
weight and the resulting graph will be of smaller order than $G$,
contradicting vertex minimality.
\end{proof}

\begin{lemma}\label{CollapsePairOfVertices}
Consider a weighted graph $(G,w)$. If there exist distinct vertices
$a_1,a_2\in A$ with $\Gamma(a_1)=\Gamma(a_2)$ then $(G,w)$ is not
vertex minimal.
\end{lemma}
\begin{proof}
Remove vertex $a_2$ and increase the weight of $a_1$ by $w(a_2)$.
The resulting weighted graph has the same densities as $(G,w)$.
\end{proof}

\begin{lemma}\label{ReplaceBy8}
Given a tripartite graph $G$ with $|A|=3$, not necessarily distinct,
vertices $a_0,a_1\in A$, $b_0,b_1\in B$ such that $a_0b_0\notin
E(G), a_1b_1\in E(G)$ and $C_{a_0b_0}=C_{a_1b_1}$, construct two
graphs $G_1, G_2$ as follows:
\begin{itemize}
\item Let $G_1'=G-a_1b_1$. Construct $G_1$ from $G_1'$ by adding a new vertex $a_2$ to $A$ and
adding edges incident to $a_2$ so that $\Gamma^{G_1}(a_2)=\Gamma^{G_1'}(a_0)\cup\{b_0\}$.
\item Let $G_2'=G+a_0b_0$. Construct $G_2$ from $G_2'$ by adding a new vertex $a_2$ to $A$ and adding edges incident to $a_2$ so that $\Gamma^{G_2}(a_2)=\Gamma^{G_2'}(a_1)\setminus\{b_1\}$.
\end{itemize}
Note that in $G_1$ and $G_2$ we have $|A|=4$.
Let $\mathcal{H}$ denote the family of eight graphs constructed from $G_1$ or $G_2$ by deleting a single vertex from $A$.

If $(G,w)$ is extremal and vertex minimal then there exists
$H\in \mathcal{H}$ and a weighting $w'$ of $H$,
such that $(H,w')$ has the same edge densities as $(G,w)$ and is
also extremal and vertex minimal.
\end{lemma}

\begin{proof}
Our proof will involve first showing that there exists a weighting
$w''$ of $G_1, G_2$ such that either $(G_1,w'')$ or $(G_2,w'')$ have
the same densities as $(G,w)$.

Form a doubly-weighted graph $(G,w,p)$ with $p(e)=1$ for all $e\in
E(G)$. Since $C_{a_0b_0}=C_{a_1b_1}$, if we add the edge $a_0b_0$ to
$G$ we can move weight from edge $a_1b_1$ to $a_0b_0$, by modifying
$p(a_1b_1)$ and $p(a_0b_0)$, whilst keeping the edge and triangle
densities constant. If we move as much weight as we can from
$a_1b_1$ to $a_0b_0$, one of two things must happen. Either we
manage to make $p(a_0b_0)=1$ before $p(a_1b_1)$ reaches $0$, or
$p(a_1b_1)$ reaches $0$ (so we remove edge $a_1b_1$) and
$p(a_0b_0)\leq 1$. In either case we have at most one partial edge
either $a_1b_1$ or $a_0b_0$. We can remove the partial edge by an
application of the $\Split$ algorithm, introducing an extra vertex
into class $A$. The two possible resulting graphs are $G_2, G_1$
respectively. Hence there exists a weighting $w''$ such that either
$(G_1,w'')$ or $(G_2,w'')$ have the same densities as $(G,w)$.

Without loss of generality let us assume $(G_1,w'')$ has the same
densities as $(G,w)$. Since $|A|=4$ for $G_1$, applying the
$\Reduce$ algorithm will remove at least one vertex from $A$ to
create a doubly-weighted graph, say $(H,w')$, with the same edge
densities and possibly a smaller triangle density. However, since
$t(G,w)=t(G_1,w'')\geq t(H,w')$ and $(G,w)$ is extremal, we must
have $t(G,w)=t(H,w')$, implying $(H,w')$ is extremal. We can also
conclude, by the vertex minimality of $(G,w)$, that $H$ is formed
from $G_1$ by removing exactly one vertex from $A$.
\end{proof}

\subsection{Search for extremal examples}\label{Sec_Search}

We have now developed a number of important conditions that any
vertex minimal extremal examples must satisfy. These will,
eventually, allow us to conduct an exhaustive search for such graphs
(with the aid of a computer). This will then leave us with a small
number of possible extremal graphs which we will deal with by hand.

\begin{figure}[tbp]
\begin{center}
\includegraphics[height=4cm]{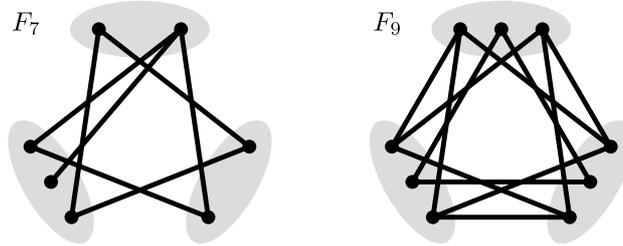}
\caption{The tripartite complements of the graphs $F_7$ and
$F_9$.}\label{Fig_F7F9}
\end{center}
\end{figure}
Recall that the tripartite graphs $G$ and $H$ (as always with
specified tripartitions) are \emph{strongly-isomorphic} if there is
a graph isomorphism $f:G\to H$ such that the image of each vertex
class in $G$ is a vertex class in $H$.

It turns out that if we can eliminate graphs that are
strongly-isomorphic to two particular examples: $F_7$ and $F_9$ (see
Figure \ref{Fig_F7F9}), then our computer search will be able to
eliminate many more possible extremal vertex minimal examples, and
thus reduce the amount of work we will finally need to do by hand.

For ease of notation we will henceforth implicitly label the
vertices and vertex classes of all figures as in Figure
\ref{Fig_Canonical}. Indices of vertices start at $1$ and increase
clockwise. Recall that the weight associated with a vertex is
indicated with a hat above the label, for example $w(b_1)$ is
represented as $\wb_1$.
\begin{figure}[tbp]
\begin{center}
\includegraphics[height=4cm]{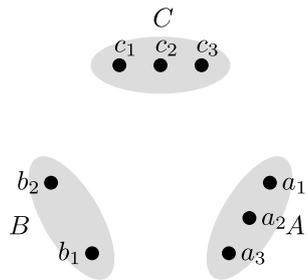}
\caption{Canonical labelling of vertices and vertex
classes.}\label{Fig_Canonical}
\end{center}
\end{figure}

\begin{lemma}\label{F7_NotExtremal}
If $(\abg)\in R_2$ then for all weightings $w$ such that
$(F_7,w)\in\Tri(\abg)$, $(F_7,w)$ is either not extremal or not
vertex minimal.
\end{lemma}

To prove Lemma \ref{F7_NotExtremal}, we first need to prove the
following result about the graph $F_6$ given in Figure \ref{Fig_F6}.
\begin{lemma}\label{F6_NotExtremal}
For any $\abg\in[0,1]$ and weighting $w$ satisfying
$(F_6,w)\in\Tri(\abg)$ we have
\begin{equation}\label{Eq_F6}
t(F_6,w)\geq \min\{\alpha\beta+\gamma-1, \alpha\gamma+\beta-1,
\beta\gamma+\alpha-1\}.
\end{equation}
\end{lemma}
\begin{figure}[tbp]
\begin{center}
\includegraphics[height=4cm]{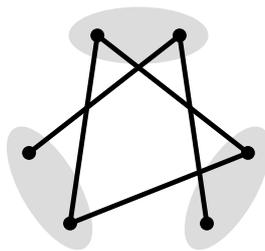}
\caption{The tripartite complement of the graph
$F_6$.}\label{Fig_F6}
\end{center}
\end{figure}
\begin{proof}
Suppose (\ref{Eq_F6}) fails to hold. Since $F_6$ contains only one
triangle: $a_2b_2c_1$, and using the fact that $\wa_2 = 1-\wa_1$,
$\wb_2= 1-\wb_1$, $\wc_2 = 1-\wc_1$, we have
\begin{align}
t(F_6,w) &= (1-\wa_1)(1-\wb_1)\wc_1\label{Eq_F6_t}\\
\alpha   &= \wb_1(1-\wc_1)+(1-\wb_1)\wc_1\label{Eq_F6_a}\\
\beta    &= \wa_1(1-\wc_1)+(1-\wa_1)\wc_1\label{Eq_F6_b}\\
\gamma   &= 1-\wa_1\wb_1\label{Eq_F6_g},
\end{align}
Substitute (\ref{Eq_F6_t}), (\ref{Eq_F6_a}), (\ref{Eq_F6_b}),
(\ref{Eq_F6_g}), into $t(F_6,w)<\alpha\beta+\gamma-1$ and rearrange
to obtain
\begin{equation}\label{Eq_F6neg}
(1-2\wa_1)(1-2\wb_1)(1-\wc_1)\wc_1+\wa_1\wb_1\wc_1<0.
\end{equation}
This implies (since $\wc_1,1-\wc_1,\wa_1,\wb_1\geq 0$) that
$0<1-2\wa_1$ or $0<1-2\wb_1$ (if $1-2\wa_1\leq 0$ and $1-2\wb_1\leq
0$ then the LHS of (\ref{Eq_F6neg}) would be non-negative).

If $0<1-2\wa_1$ is true then substitute (\ref{Eq_F6_t}),
(\ref{Eq_F6_a}), (\ref{Eq_F6_b}), (\ref{Eq_F6_g}), into
$t(F_6,w)<\alpha\gamma+\beta-1$ and rearrange to obtain
\[
\wa_1\wb_1\wc_1(2-\wb_1) + \wa_1(1-\wb_1)^2(1-\wc_1) +
(1-2\wa_1)(1-\wb_1)(1-\wc_1)<0.
\]
But each term in the LHS is strictly non-negative so we have a
contradiction.

If instead $0<1-2\wb_1$ holds then looking at
$t(F_6,w)<\beta\gamma+\alpha-1$ yields
\[
\wa_1\wb_1\wc_1(2-\wa_1) + \wb_1(1-\wa_1)^2(1-\wc_1) +
(1-2\wb_1)(1-\wa_1)(1-\wc_1)<0,
\]
which is similarly false.
\end{proof}

\begin{proof}[Proof of Lemma \ref{F7_NotExtremal}]
Suppose $(F_7,w)$ is extremal and vertex minimal. We may assume
$w(v)\in (0,1)$ for all $v\in V(F_7)$. If
$t(F_7,w)\geq\alpha\beta+\gamma-1$ then by Lemma
\ref{CyclicIneq_NotExtremal} $(F_7,w)$ is not extremal, so we may
assume that
\begin{align}
t(F_7,w) &< \alpha\beta+\gamma-1, \label{Eq_F7_ab+g-1}
\intertext{and similarly}
t(F_7,w) &< \alpha\gamma+\beta-1, \label{Eq_F7_ag+b-1}\\
t(F_7,w) &< \beta\gamma+\alpha-1. \label{Eq_F7_bg+a-1}
\end{align}
Consider moving all the weight from $b_3$ to $b_2$ to create the
following weighting $w'$ of $F_7$ defined formally as $w'(v)=w(v)$
for all $v\in V(G)\setminus\{b_2,b_3\}$, $w'(b_2)=w(b_2)+w(b_3)$,
and $w'(b_3)=0$. Changing the weighting from $w$ to $w'$ does not
change the edge density between $A$ and $C$, or $B$ and $C$, but it
may have increased the edge density between $A$ and $B$ and the
triangle density. Let us call the new edge density, between $A$ and
$B$, $\gamma'$. Its value can be expressed in terms of the old
weights and densities
\[
\gamma' = \gamma+\wa_2\wb_3.
\]
Similarly
\[
t(F_7,w') = t(F_7,w)+\wa_2\wb_3\wc_1.
\]

If we can show that
\begin{align}
t(F_7,w') &< \alpha\beta+\gamma'-1, \label{Eq_F7_ab+g-1_Prime}\\
t(F_7,w') &< \alpha\gamma'+\beta-1, \label{Eq_F7_ag+b-1_Prime}\\
t(F_7,w') &< \beta\gamma'+\alpha-1, \label{Eq_F7_bg+a-1_Prime}
\end{align}
all hold then, since $w'(b_3)=0$, we could remove $b_3$ from $F_7$
leaving all densities unchanged, and the resulting graph would be
strongly-isomorphic to $F_6$. This contradicts Lemma
\ref{F6_NotExtremal}, hence our assumption that $(F_7,w)$ is
extremal and vertex minimal must be false.

First let us show that (\ref{Eq_F7_ab+g-1_Prime}) holds. Consider
\begin{align*}
\alpha\beta+\gamma'-1-t(F_7,w') &=
\alpha\beta+(\gamma+\wa_2\wb_3)-1-(t(F_7,w)+\wa_2\wb_3\wc_1)\\
&= \alpha\beta+\gamma-1-t(F_7,w)+\wa_2\wb_3(1-\wc_1)\\
&>0.
\end{align*}
The inequality holds because $\alpha\beta+\gamma-1-t(F_7,w)>0$ by
(\ref{Eq_F7_ab+g-1}) and $\wa_2,\wb_3,\wc_1\in (0,1)$.

To prove (\ref{Eq_F7_ag+b-1_Prime}) we look at
\begin{align*}
\alpha\gamma'+\beta-1-t(F_7,w') &=
\alpha(\gamma+\wa_2\wb_3)+\beta-1-(t(F_7,w)+\wa_2\wb_3\wc_1)\\
&= \alpha\gamma+\beta-1-t(F_7,w)+\wa_2\wb_3(\alpha-\wc_1).
\end{align*}
We know $\alpha\gamma+\beta-1-t(F_7,w)>0$ by (\ref{Eq_F7_ag+b-1}),
and $\wa_2,\wb_3>0$, so all we have to do is show that
$\alpha-\wc_1\geq 0$. By definition $\alpha$ is the sum of the
weighted edges between $B$ and $C$, hence
\begin{align*}
\alpha &= (\wb_2+\wb_3)\wc_1+\wb_1\wc_2\\
&= (1-\wb_1)\wc_1+\wb_1(1-\wc_1).
\end{align*}
Therefore
\begin{align*}
\alpha-\wc_1 &= (1-\wb_1)\wc_1+\wb_1(1-\wc_1)-\wc_1\\
&= \wb_1(1-2\wc_1).
\end{align*}
Since $\wb_1$ is greater than $0$, we require $\wc_1\leq 1/2$.

Consider $C_{a_1b_1}=\{c_2\}$ and $C_{a_2b_2}=\{c_1\}$. Construct
$(F_7,w,p)\in\DTri$ by setting $p(e)=1$ for all edges of $F_7$. If
$\wc_2<\wc_1$ then, by Lemma \ref{Order_Sset}, we know we can
achieve a smaller triangle density. Therefore $\wc_1\leq\wc_2$ must
hold, or equivalently $\wc_1\leq 1/2$ (as $\wc_1+\wc_2=1$).

Similarly to prove (\ref{Eq_F7_bg+a-1_Prime}) consider
\begin{align*}
\beta\gamma'+\alpha-1-t(F_7,w') =
\beta\gamma+\alpha-1-t(F_7,w)+\wa_2\wb_3(\beta-\wc_1).
\end{align*}
By (\ref{Eq_F7_bg+a-1}) we need only show $\beta-\wc_1\geq 0$, which
is true because $\beta-\wc_1 = \wa_1(1-2\wc_1)$, $\wa_1>0$, and
$\wc_1\leq 1/2$.
\end{proof}

\begin{lemma}\label{F9_NotExtremal}
For all weightings $w$ such that $(F_9,w)\in\Tri$, $(F_9,w)$ is
either not extremal or not vertex minimal.
\end{lemma}

\begin{proof}
Let us assume that $(F_9,w)$ is extremal and vertex minimal, in
which case $w(v)\neq 0$ for all $v\in V(F_9)$. Construct
$(F_{10},w')\in\Tri$ from $(F_9,w)$ as follows:
\begin{itemize}
\item Create $F_{10}$ from $F_9$ by removing the edge $a_3c_1$.
Add a new vertex into $C$, labelled $c_4$, and add in edges so that
$\Gamma^{F_{10}}(c_4)=\Gamma_B^{F_9}(c_1)\cup A$.
\item Set $w'(v)=w(v)$ for all $v\in V(F_{10})\setminus\{c_1,c_4\}$. Let
\[
w'(c_1)=\frac{w(a_1)w(c_1)}{w(a_1)+w(a_3)},\qquad\mbox{and}\qquad
w'(c_4)=\frac{w(a_3)w(c_1)}{w(a_1)+w(a_3)}.
\]
\end{itemize}
The edge density between $A$ and $B$ remains unchanged and it is
easy to check that the density between $B$ and $C$ also hasn't
changed. The change in edge density between $A$ and $C$ is
\begin{align*}
w(a_2)w'(c_1)+w'(c_4)-w(a_2)w(c_1)-w(a_3)w(c_1) = 0.
\end{align*}
The triangles in $F_9$ are $a_1b_3c_2$, $a_2b_1c_3$, $a_3b_2c_1$ and
the triangles in $F_{10}$ are $a_1b_3c_2$, $a_2b_1c_3$, $a_1b_2c_4$,
$a_3b_2c_4$. Hence the change in triangle density between $(F_9,w)$
and $(F_{10},w')$ is
\[
\bigl(w(a_1)+w(a_3)\bigr)w(b_2)w'(c_4)-w(a_3)w(b_2)w(c_1) = 0.
\]
Therefore $(F_9,w)$ and $(F_{10},w')$ have the same triangle and
edge densities.

Note that $\Gamma_C^{F_{10}}(a_1)=\Gamma_C^{F_1{10}}(a_3)=\{c_2,c_4\}$.
Since $|C|=4$ we can apply the $\Reduce$ algorithm to class $C$ in $F_{10}$, and
the resultant output $(F'',w'')\in\Tri$ has the same edge densities
and the same triangle density (because $(F_9,w)$ is extremal).
Moreover $|V(F'')|=|V(F_9)|$ (as $(F_9,w)$ is vertex minimal)
and $\Gamma_C^{F''}(a_1)=\Gamma_C^{F''}(a_3)$. Hence we can apply Lemma
\ref{OppositeClassSizeIs3} to $(F'',w'')$, showing that it is
either not extremal or not vertex minimal and so the same must be
true of $(F_9,w)$.
\end{proof}

Our goal is to produce a list of all tripartite graphs $G$ for which
there exists a weighting $w$ such that $(G,w)\in\Tri(\abg)$ is
extremal and vertex minimal for some $(\abg)\in R_2$. With this aim
in mind we have developed a number of results that allow us to show
$(G,w)$ is not extremal or not vertex minimal by simply examining
$G$, \emph{irrespective of the weighting $w$}.

By Lemmas \ref{ClassSize1_NotExtremal} and \ref{MaxClassSize3} we
need only consider tripartite graphs $G$ in which all vertex classes
contain either two or three vertices. This reduces the problem to a
finite search. However, tripartite graphs with $|A|=|B|=|C|=3$ can
contain $27$ possible edges, so naively there are at least $2^{27}
\approx 100,000,000$ graphs to consider. We can decrease the
possible number of graphs by looking at only those that contain
triangles, since otherwise $(\abg)\notin R$ by Theorem
\ref{Thm_T=0}. By Lemma \ref{CollapsePairOfVertices} we know that if
$G$ has a class containing a pair of vertices with identical
neighbours then it is not vertex minimal (because we can move all
the weight from one vertex to the other). Similarly the more
technical results given in Corollary \ref{Order_subset}, Lemmas
\ref{CollapseClass}, \ref{OppositeClassSizeIs3}, \ref{ReplaceBy8},
\ref{F7_NotExtremal}, and \ref{F9_NotExtremal} can also be used to
eliminate graphs without knowledge of the vertex weights. Tripartite
graphs that are strongly-isomorphic to graphs eliminated by these
results will also not be extremal or not vertex minimal, and so may
also be discarded.

Unfortunately applying Corollary \ref{Order_subset}, Lemmas
\ref{CollapseClass}, \ref{OppositeClassSizeIs3},
\ref{CollapsePairOfVertices}, \ref{ReplaceBy8},
\ref{F7_NotExtremal}, \ref{F9_NotExtremal} and Theorem \ref{Thm_T=0}
to over $100,000,000$ tripartite graphs would take too long to
perform by hand, but can easily be done by computer. A C++
implementation is given in the Appendix. This algorithm produces a
list of possible extremal vertex minimal tripartite graphs in $R_2$,
which are equivalent up to strong-isomorphism to the fourteen
graphs, given in Figure \ref{Fig_Program_Output}. To decrease the
number further we will have to check each of these graphs by hand.

\begin{figure}[tbp]
\begin{center}
\includegraphics[width=12cm]{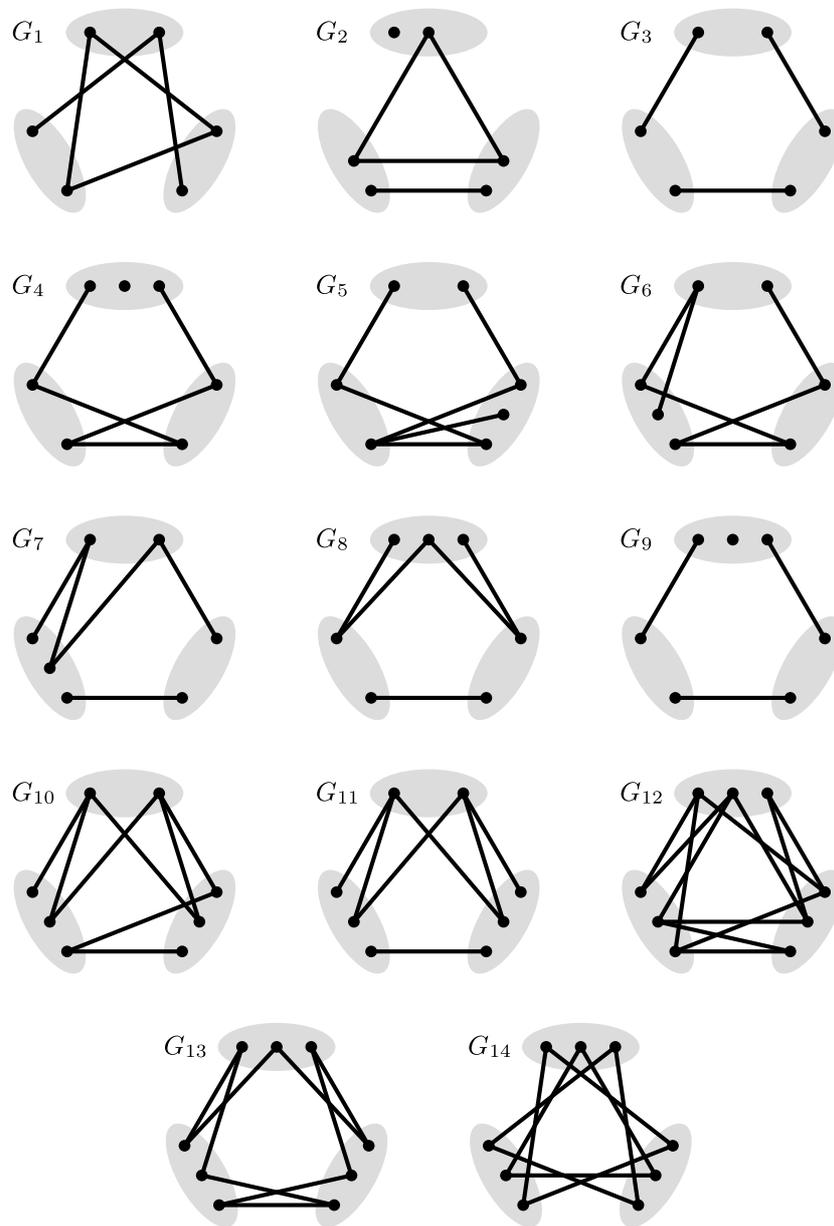} 
\caption{The tripartite complements of the graphs produced by the
computer search.}\label{Fig_Program_Output}
\end{center}
\end{figure}

\subsection{Specific Graphs}\label{Sec_SpecificGraphs}
To complete the proof of Theorem \ref{Thm_H7H7'H9} we need to
eliminate the eleven graphs found by the computer search, other than
$H_7,H_7'$ and $H_9$. (In the list of fourteen graphs these are
$\Gviii,\Gvii$ and $\Gxiii$ respectively.)

To be precise we will show that for each $G_i$, $1\leq i\leq 14$,
$i\neq 7,8,13$, if $(\abg)\in R_2$ then there does not exist a
weighting $w$ such that $(G_i,w)\in \Tri(\abg)$ and $(G_i,w)$ is
both extremal and vertex minimal.

\begin{lemma}\label{G1_NotExtremal}
If $(\abg)\in R_2$ then for all weightings $w$ such that
$(\Gi,w)\in\Tri(\abg)$, $(\Gi,w)$ is not extremal.
\end{lemma}

\begin{proof}
$\Gi$ is strongly-isomorphic to $F_6$. Hence Lemma
\ref{F6_NotExtremal} and Lemma \ref{CyclicIneq_NotExtremal} imply
$(\Gi,w)$ is not extremal.
\end{proof}

\begin{lemma}
If $(\abg)\in R_2$ then for all weightings $w$ such that
$(\Gii,w)\in\Tri(\abg)$, $(\Gii,w)$ is not extremal.
\end{lemma}

\begin{proof}
Suppose $(\Gii,w)$ is extremal, by Lemma
\ref{CyclicIneq_NotExtremal} we must have
$t(\Gii,w)<\alpha\beta+\gamma-1$. The edge and triangle densities
are given by
\begin{align*}
\alpha &= \wb_1+(1-\wb_1)\wc_1,\\
\beta &= 1-\wa_1+\wa_1\wc_1,\\
\gamma &= \wa_1\wb_1+(1-\wa_1)(1-\wb_1),\\
t(\Gii,w) &= \wa_1\wb_1\wc_1+(1-\wa_1)(1-\wb_1)\wc_1.
\end{align*}
Substituting into $t(\Gii,w)<\alpha\beta+\gamma-1$ and simplifying
yields
\[
\wa_1(1-\wb_1)(1-\wc_1)(1+\wc_1)<0
\]
which is false.
\end{proof}

\begin{lemma}
For $(\abg)\in R_2$ there exist no weightings $w$ of $\Giii$ such
that $(\Giii,w)\in\Tri(\abg)$.
\end{lemma}

\begin{proof}
$\Giii$ is strongly-isomorphic to $H_6$. Hence the result follows
immediately from Lemma \ref{6Vertex_ValidRegion}.
\end{proof}

\begin{lemma}\label{G4_NotExtremal}
If $(\abg)\in R_2$ then for all weightings $w$ such that
$(\Giv,w)\in\Tri(\abg)$, $(\Giv,w)$ is either not extremal or not
vertex minimal.
\end{lemma}

\begin{proof}
Let us assume $(\Giv,w)\in\Tri(\abg)$ is vertex minimal, and so
$w(v)\neq 0$ for all $v\in V(\Giv)$. By Lemma
\ref{abgInR3_implies_Not01} (ii) we also have $\abg\neq 0,1$. The
densities in terms of the vertex weights $\wa_1, \wb_2, \wc_1,$ and
$\wc_3$, are as follows,
\begin{align*}
\alpha &= 1-\wb_2\wc_1,\\
\beta  &= 1-\wa_1\wc_3,\\
\gamma &= \wa_1\wb_2,\\
t(\Giv,w) &= \wa_1\wb_2(1-\wc_1-\wc_3).
\end{align*}
We can use these equations to write $\wa_1,\wb_2,\wc_3,$ and
$t(\Giv,w)$ in terms of $\wc_1$,
\begin{align}
\wb_2 &= \frac{1-\alpha}{\wc_1}, \label{Eq_G4_bc}\\
\wa_1 &= \frac{\gamma\wc_1}{1-\alpha}, \label{Eq_G4_ac}\\
\wc_3 &= \frac{(1-\alpha)(1-\beta)}{\gamma\wc_1}, \notag\\
t(\Giv,w) &= \gamma-\gamma\wc_1-\frac{(1-\alpha)(1-\beta)}{\wc_1}.
\label{Eq_G4_tc}
\end{align}

From (\ref{Eq_G4_tc}) we can deduce that $t(\Giv,w)$ will be
minimized when $\wc_1$ is as large or as small as possible, because
the second derivative with respect to $\wc_1$ is negative. Since
$\wb_2\leq 1$ and $\wa_1\leq 1$, (\ref{Eq_G4_bc}) and
(\ref{Eq_G4_ac}) imply that $\wc_1\in [1-\alpha,(1-\alpha)/\gamma]$.

Substituting $\wc_1=1-\alpha$ into (\ref{Eq_G4_tc}) gives
$t(\Giv,w)=\alpha\gamma+\beta-1$. Substituting
$\wc_1=(1-\alpha)/\gamma$ into (\ref{Eq_G4_tc}) gives
$t(\Giv,w)=\beta\gamma+\alpha-1$. Hence for
$\wc_1\in[1-\alpha,(1-\alpha)/\gamma]$ we have
\[
t(\Giv,w)\geq \min\{\alpha\gamma+\beta-1,\beta\gamma+\alpha-1\}.
\]
Lemma \ref{CyclicIneq_NotExtremal} therefore tells us that
$(\Giv,w)$ can not be extremal.
\end{proof}

\begin{lemma}\label{G5_NotExtremal}
If $(\abg)\in R_2$ then for all weightings $w$ such that
$(\Gv,w)\in\Tri(\abg)$, $(\Gv,w)$ is either not extremal or not
vertex minimal.
\end{lemma}

\begin{proof}
Suppose $(\abg)\in R_2$ and $(\Gv,w)\in\Tri(\abg)$. We will show
that there exists a weighting $w'$ of $\Giv$ such that
$(\Giv,w')\in\Tri(\abg)$ and $t(\Giv,w')=t(\Gv,w)$. Since
$|V(\Giv)|=|V(\Gv)|$, Lemma \ref{G4_NotExtremal} implies that
$(\Gv,w)$ is either not extremal or not vertex minimal.

Suppose $(\Gv,w)$ is vertex minimal, in which case we may assume
$w(v)>0$ for all $v\in V(\Gv)$. To prove there exists $(\Giv,w')$
with the same densities as $(\Gv,w)$, note that
$\Gamma_B(a_1)=\Gamma_B(a_2)$ in $\Gv$. Hence we can modify $\Gv$ by
applying $\Merge$ on $a_1, a_2$ labelling the resulting merged
vertex by $a$. This creates one partial edge $ac_2$. Apply $\Split$
on this edge, to remove it, choosing to replace the vertex $c_2$.
The resulting weighted tripartite graph has the same densities as
$(\Gv,w)$ and it is easy to check that it is strongly-isomorphic to
$\Giv$.
\end{proof}

\begin{lemma}\label{G6_NotExtremal}
If $(\abg)\in R_2$ then for all weightings $w$ such that
$(\Gvi,w)\in\Tri(\abg)$, $(\Gvi,w)$ is either not extremal or not
vertex minimal.
\end{lemma}

\begin{proof}
Suppose $(\abg)\in R_2$ and $(\Gvi,w)\in\Tri(\abg)$. We will show
that there exists a weighting $w'$ of $\Gv$ such that
$(\Gv,w')\in\Tri(\abg)$ and $t(\Gv,w')=t(\Gvi,w)$. Since
$|V(\Gv)|=|V(\Gvi)|$, Lemma \ref{G5_NotExtremal} implies that
$(\Gvi,w)$ is either not extremal or not vertex minimal.

Suppose $(\Gvi,w)$ is vertex minimal, in which case we may assume
$w(v)>0$ for all $v\in V(\Gvi)$. To prove there exists $(\Gv,w')$
with the same densities as $(\Gvi,w)$, note that
$\Gamma_C(b_2)=\Gamma_C(b_3)$ in $\Gvi$. Hence we can modify $\Gvi$
by applying $\Merge$ on $b_2, b_3$ labelling the resulting merged
vertex $b$. This creates one partial edge $a_2b$. Apply $\Split$ on
that edge, to remove it, choosing to replace the vertex $a_2$. The
resulting weighted tripartite graph has the same densities as
$(\Gvi,w)$ and it is easy to check that it is strongly-isomorphic to
$\Gv$.
\end{proof}

\begin{lemma}
For $(\abg)\in R_2$ there exist no weightings $w$ of $\Gix$ such
that $(\Gix,w)\in\Tri(\abg)$.
\end{lemma}

\begin{proof}
Suppose $(\Gix,w)\in\Tri(\abg)$ for $(\abg)\in R_2$. If $w(c_2)=0$
then removing $c_2$ leaves $\Gix$ strongly-isomorphic to $H_6$.
Hence we get a contradiction from Lemma \ref{6Vertex_ValidRegion}.
If $w(c_1)=0$ or $w(b_2)=0$ then $\alpha=1$, and
$(1,\beta,\gamma)\notin R_2$ by Lemma \ref{abgInR3_implies_Not01}
(ii). Similarly we can show all other vertices must have a non-zero
weight. We will get a contradiction by showing that
$\Delta(\abg)\geq 0$ and hence $(\abg)\notin R_2$.

Consider a new weighting $w'$ given by $w'(v)=w(v)$ for all $v\in
V(\Gix)\setminus\{c_1,c_2\}$, $w'(c_1)=w(c_1)+w(c_2)$, and
$w'(c_2)=0$. For convenience let us write $\alpha'=\alpha(\Gix,w')$
(note that $\beta(\Gix,w')=\beta$ and $\gamma(\Gix,w')=\gamma$).
Since $w'(c_2)=0$ we could remove it from $\Gix$ without changing
any densities and the resulting weighted tripartite graph would be
strongly-isomorphic to $H_6$, let $w''$ be the corresponding
weighting. Since $w(v)\neq 0$ for all $v\in V(\Gix)$ we know
$w''(v)\neq 0$ for all $v\in V(H_6)$, and consequently $t(H_6,w'')>
0$. Lemma \ref{6Vertex_IsBest} tells us that
$T_{\min}(\alpha',\beta,\gamma)=t(H_6,w'')> 0$, therefore by Theorem
\ref{Thm_T=0} we have $(\alpha',\beta,\gamma)\in R$. Moreover, Lemma
\ref{6Vertex_ValidRegion} implies that
$\Delta(\alpha',\beta,\gamma)\geq 0$.

Since $\alpha'=1-w'(b_2)w'(c_1)=1-w(b_2)w(c_1)-w(b_2)w(c_2)$, we
have
\[
\alpha' = \alpha-w(b_2)w(c_2).
\]
Hence we can write $\alpha = \alpha'+\epsilon$, where
$\epsilon=w(b_2)w(c_2)>0$. Consider
\begin{align*}
\Delta(\abg) &= \Delta(\alpha'+\epsilon,\beta,\gamma)\\
&= \Delta(\alpha',\beta,\gamma) + 2\epsilon\alpha' + \epsilon^2 - 2\epsilon\beta - 2\epsilon\gamma + 4\epsilon\beta\gamma\\
&= \Delta(\alpha',\beta,\gamma) + \epsilon^2 + 2\epsilon(\alpha'+\beta+\gamma-2) + 4\epsilon(1-\beta)(1-\gamma)\\
&=\Delta(\alpha',\beta,\gamma) + \epsilon^2 + 2\epsilon t(H_6,w'') +
4\epsilon(1-\beta)(1-\gamma).
\end{align*}
Since each term is non-negative we have
$\Delta(\alpha,\beta,\gamma)\geq 0$. Therefore $(\abg)\notin R_2$, a
contradiction.
\end{proof}

\begin{lemma}\label{G10_NotExtremal}
For all weightings $w$ such that $(\Gx,w)\in\Tri$, $(\Gx,w)$ is
either not extremal or not vertex minimal.
\end{lemma}

\begin{proof}
Suppose $(\Gx,w)$ is extremal and vertex minimal, hence $w(v)\neq 0$
for all $v\in V(\Gx)$. Convert $(\Gx,w)$ into a doubly-weighted
tripartite graph by adding the function $p$ which maps all edges to
$1$. Applying $\Merge$ on $(\Gx,w,p)$ and $b_2,b_3$, results in only
one partial edge being created $bc_2$ (where $b$ is the vertex
replacing $b_2,b_3$). We can apply $\Split$ on that edge choosing to
replace the vertex $c_2$, and then revert back to a weighted graph
$(\Gx',w')$ say. Now $(\Gx',w')$ has the same densities as $(\Gx,w)$
but $\Gx'$ has $|B|=2$ and $|C|=3$. Moreover, $\Gx$ has the property
that $\Gamma_B(a_1)=\Gamma_B(a_3)$, and this is also true in $\Gx'$.
Hence applying Lemma \ref{OppositeClassSizeIs3} to $(\Gx',w')$ and
$a_1,a_3$, we see that $(\Gx',w')$ is not extremal or not vertex
minimal. Since $|V(\Gx')|=|V(\Gx)|$ the same is true of $(\Gx,w)$.
\end{proof}

\begin{lemma}
For all weightings $w$ such that $(\Gxi,w)\in\Tri$, $(\Gxi,w)$ is
either not extremal or not vertex minimal.
\end{lemma}

\begin{proof}
The proof is almost identical to that of Lemma
\ref{G10_NotExtremal}. The only difference being at the end, where
now we have $\Gamma_B(a_1)=\Gamma_B(a_2)$ holding true, and so we
apply Lemma \ref{OppositeClassSizeIs3} to vertices $a_1$ and $a_2$
instead.
\end{proof}

\begin{lemma}
For all weightings $w$ such that $(\Gxii,w)\in\Tri$, $(\Gxii,w)$ is
either not extremal or not vertex minimal.
\end{lemma}

\begin{proof}
Suppose $(\Gxii,w)$ is vertex minimal, so $w(v)> 0$ for all $v\in
V(\Gxii)$. Of the three statements $\wa_1\leq\wa_2$,
$\wb_1\leq\wb_2$, $\wc_1\leq\wc_2$, at least two must be true or at
least two must be false. Without loss of generality let us suppose
that $\wa_1\leq\wa_2$, $\wb_1\leq\wb_2$ are both true.

The densities of $(\Gxii,w)$ are given by
\begin{align*}
t(\Gxii,w) &= \wa_3\wb_3\wc_3,\\
\alpha &= \wb_1\wc_2+\wb_2\wc_1+\wc_3,\\
\beta  &= \wa_1\wc_2+\wa_2\wc_1+\wa_3,\\
\gamma &= \wa_1\wb_2+\wa_2\wb_1+\wb_3.
\end{align*}

Consider the doubly-weighted tripartite graph $(\Gxii,w,p)$ where
$p$ maps all edges to $1$. It has the same densities as $(\Gxii,w)$.
If we move a sufficiently small amount of weight $\delta>0$ from
vertex $c_2$ to $c_1$, $\alpha$ and $\beta$ increase. By decreasing
$p(b_3c_3)$ and $p(a_3c_3)$ respectively we can keep all densities
unchanged. More precisely set
\[
p(b_3c_3)=1-\delta(\wb_2-\wb_1)/\wb_3\wc_3,\qquad
p(a_3c_3)=1-\delta(\wa_2-\wa_1)/\wa_3\wc_3.
\]
If $\wa_1=\wa_2$ and $\wb_1=\wb_2$, then increasing the weight of
$c_1$ to $\wc_1+\wc_2$ and removing $c_2$ will result in a weighted
tripartite graph with the same densities as $(\Gxii,w)$ but with
fewer vertices. Hence we know that $p(b_3c_3)<1$ or $p(a_3c_3)<1$.
Consequently we now have a doubly-weighted tripartite graph with the
same edge densities as $(\Gxii,w)$ but a strictly smaller triangle
density. Hence by Lemma \ref{DoublyWeighted_SameAs_Weighted}
$(\Gxii,w)$, is not extremal.

Suppose now that two of the statements $\wa_1\leq\wa_2$,
$\wb_1\leq\wb_2$, $\wc_1\leq\wc_2$, are false, for example
$\wa_1>\wa_2$ and $\wb_1>\wb_2$. We can repeat the above argument,
this time moving weight from $c_1$ to $c_2$, again constructing a
doubly-weighted tripartite graph with the same edge densities but a
smaller triangle density.
\end{proof}

\begin{lemma}\label{G14_NotExtremal}
If $(\abg)\in R_2$ then for all weightings $w$ such that
$(\Gxiv,w)\in\Tri(\abg)$, $(\Gxiv,w)$ is either not extremal or not
vertex minimal.
\end{lemma}

\begin{proof}
Suppose $(\Gxiv,w)$ is extremal and vertex minimal, so $w(v)>0$ for
all $v\in V(\Gxiv)$. Consider the doubly-weighted tripartite graph
$(\Gxiv,w,p)$, where $p$ maps all edges to $1$. Applying Lemma
\ref{Order_Sset} to $(\Gxiv,w,p)$ on the non-edge $a_1b_1$ and the
edge $a_3b_2$ tells us that in order to be extremal
\[
\sum_{c\in C_{a_1b_1}} w(c)\geq \sum_{c\in C_{a_3b_2}} w(c)
\]
must hold. Since $C_{a_1b_1}=\{c_2,c_3\}$ and $C_{a_3b_2}=\{c_1\}$
we must have $\wc_2+\wc_3\geq\wc_1$ or equivalently $1-2\wc_1\geq 0$
(using the fact that $\wc_1+\wc_2+\wc_3=1$). Similarly we can show
that $1-2\wc_2\geq 0$ by looking at $a_2b_2$, $a_1b_3$, and
$1-2\wc_3\geq 0$ by taking $a_3b_3$, $a_2b_1$. By symmetry we must
have $1-2w(v)\geq 0$ for all $v\in V(\Gxiv)$. Note that the function
$w'$ defined by $w'(v)=1-2w(v)$ for all $v\in V(\Gxiv)$ provides a
valid weighting of $\Gxiv$, as $w'(v)\geq 0$ for all $v\in V(\Gxiv)$
and the sum of the weights in a class, $X$ say, is
\begin{align*}
\sum_{v\in X}w'(v) &=  \sum_{v\in X}(1-2w(v))\\
&= |X|-2\sum_{v\in X}w(v)\\
&= |X|-2\\
&=1
\end{align*}
because every class in $\Gxiv$ has size $3$.

Let $\Gbar$ be the tripartite complement of the graph $\Gxiv$.
Consider the weighted tripartite graph $(\Gbar,w')$, with edge
densities
\[
\alpha(\Gbar,w')=\alpha', \quad \beta(\Gbar,w')=\beta', \quad
\gamma(\Gbar,w')=\gamma'.
\]
We can write down $\alpha'$ in terms of $\alpha$.
\begin{align*}
\alpha' &= (1-2\wb_1)(1-2\wc_1) + (1-2\wb_2)(1-2\wc_2) +
(1-2\wb_3)(1-2\wc_3)\\
&= 3 - 2(\wb_1+\wb_2+\wb_3) - 2(\wc_1+\wc_2+\wc_3) +
4(\wb_1\wc_1+\wb_2\wc_2+\wb_3\wc_3)\\
&= 3 - 4(1-\wb_1\wc_1-\wb_2\wc_2-\wb_3\wc_3)\\
&= 3-4\alpha,
\end{align*}
similarly $\beta'=3-4\beta$, and $\gamma'=3-4\gamma$. Next let us
write $t(\Gbar,w')$ in terms of $t(\Gxiv,w)$,
\begin{align*}
t(\Gbar,w') &= (1-2\wa_1)(1-2\wb_1)(1-2\wc_1) + (1-2\wa_2)(1-2\wb_2)(1-2\wc_2)+\\
&\quad\,\, (1-2\wa_3)(1-2\wb_3)(1-2\wc_3)\\
&= 3 - 2(\wa_1+\wa_2+\wa_3) - 2(\wb_1+\wb_2+\wb_3) - 2(\wc_1+\wc_2+\wc_3) +\\
&\quad\,\, 4(\wa_1\wb_1+\wa_1\wc_1+\wb_1\wc_1+\wa_2\wb_2+\wa_2\wc_2+\wb_2\wc_2+\wa_3\wb_3+\wa_3\wc_3+\wb_3\wc_3)-\\
&\quad\,\, 8(\wa_1\wb_1\wc_1+\wa_2\wb_2\wc_2+\wa_3\wb_3\wc_3)\\
&= 1 + 4(\wa_1\wb_1+\wa_1\wc_1+\wb_1\wc_1+\wa_2\wb_2+\wa_2\wc_2+\wb_2\wc_2+\wa_3\wb_3+\wa_3\wc_3+\wb_3\wc_3-\\
&\quad\,\, 2\wa_1\wb_1\wc_1-2\wa_2\wb_2\wc_2-2\wa_3\wb_3\wc_3-\wa_1-\wa_2-\wa_3)\\
&= 1 + 4((1-\wa_1)\wb_1\wc_1 + (1-\wa_2)\wb_2\wc_2 + (1-\wa_3)\wb_3\wc_3-\\
&\quad\,\, \wa_1(1-\wb_1)(1-\wc_1) -\wa_2(1-\wb_2)(1-\wc_2) -\wa_3(1-\wb_3)(1-\wc_3))\\
&= 1+4((\wa_2+\wa_3)\wb_1\wc_1 + (\wa_1+\wa_3)\wb_2\wc_2 + (\wa_1+\wa_2)\wb_3\wc_3-\\
&\quad\,\, \wa_1(\wb_2+\wb_3)(\wc_2+\wc_3) - \wa_2(\wb_1+\wb_3)(\wc_1+\wc_3) - \wa_3(\wb_1+\wb_2)(\wc_1+\wc_2))\\
&= 1+4(-\wa_1\wb_2\wc_3-\wa_1\wb_3\wc_2 - \wa_2\wb_1\wc_3-\wa_2\wb_3\wc_1 - \wa_3\wb_1\wc_2-\wa_3\wb_2\wc_1)\\
&= 1-4t(\Gxiv,w).
\end{align*}
Without loss of generality suppose $\alpha'\leq\beta'\leq\gamma'$.
Since $(\Gxiv,w)$ is extremal by Lemma \ref{7Vertex_MinDensity} we
have
\[
t(\Gxiv,w)\leq 2\sqrt{\alpha\beta(1-\gamma)}+2\gamma-2.
\]
Rewriting in terms of $\alpha',\beta',\gamma',t(\Gbar,w')$ gives
\[
3+2\gamma'-t(\Gbar,w')\leq\sqrt{(3-\alpha')(3-\beta')(1+\gamma')}.
\]
Note that in any weighted tripartite graph the triangle density is
bounded above by all of the edge densities, thus $t(\Gbar,w')\leq
\alpha'$,
and so
\[
3+2\gamma'-\alpha'\leq\sqrt{(3-\alpha')(3-\beta')(1+\gamma')}.
\]
Squaring both sides and rearranging yields
\begin{align*}
\alpha'^2+\gamma'(4\gamma'-\alpha'\beta')+\gamma'(3\beta'-\alpha')+3(\gamma'-\alpha')+\beta'(3-\alpha')\leq
0.
\end{align*}
Each term is non-negative (because
$0\leq\alpha'\leq\beta'\leq\gamma'\leq 1$), and so the only way this
can be true is if $\alpha'=\beta'=\gamma'=0$. Hence
$\alpha=\beta=\gamma=3/4$, but such values do not lie in $R_2$ due
to the fact that $\Delta(3/4,3/4,3/4)=0$. Thus we have a
contradiction and our assumption that $(\Gxiv,w)$ is extremal and
vertex minimal must be false.
\end{proof}

\begin{proof}[Proof of Theorem \ref{Thm_H7H7'H9}] Our computer
search tells us that the only possible extremal and vertex minimal
tripartite graphs are strongly-isomorphic to those given in Figure
\ref{Fig_Program_Output}. Given $(\abg)\in R_2$ for all weightings
$w$, $(\Gi,w)$, $(\Gii,w)$, $(\Giii,w)$, $(\Giv,w)$, $(\Gv,w)$,
$(\Gvi,w)$, $(\Gix,w)$, $(\Gx,w)$, $(\Gxi,w)$, $(\Gxii,w)$,
$(\Gxiv,w)$ are either not extremal, not vertex minimal, or do not
lie in $\Tri(\abg)$ by Lemmas \ref{G1_NotExtremal} to
\ref{G14_NotExtremal} respectively. This just leaves $\Gvii$,
$\Gviii$, and $\Gxiii$ which are strongly-isomorphic to $H_7'$,
$H_7$ and $H_9$ respectively.
\end{proof}

\section{Conjectures}\label{Sec_Conjectures}
The following conjecture, if true, would allow us to write
$T_{\min}(\abg)$ as a simple expression for all values of
$\abg\in[0,1]$.
\begin{conjecture}\label{Conj_Tmin}
For $\gamma\leq\alpha,\beta$,
\[
T_{\min}(\abg) =
\begin{cases}
0, &\text{if $(\abg)\in [0,1]^3\setminus R$},\\
2\sqrt{\alpha\beta(1-\gamma)}+2\gamma-2, &\text{if $(\abg)\in R_2$},\\
\alpha+\beta+\gamma-2, &\text{otherwise}.
\end{cases}
\]
\end{conjecture}

To prove Conjecture \ref{Conj_Tmin} it is sufficient to prove the
subsequent conjecture.

\begin{conjecture}\label{Conj_H9_NotExtremal}
If $(\abg)\in R_2$ then for all weightings $w$ such that
$(H_9,w)\in\Tri(\abg)$, $(H_9,w)$ is either not extremal or not
vertex minimal.
\end{conjecture}

\begin{theorem}
Conjecture \ref{Conj_H9_NotExtremal} implies Conjecture
\ref{Conj_Tmin}.
\end{theorem}

\begin{proof}
Theorems \ref{Thm_T=0} and \ref{Thm_H6} tell us when
$T_{\min}(\abg)=0$ and $\alpha+\beta+\gamma-2$ respectively. By
Theorem \ref{Thm_H7H7'H9} and Conjecture \ref{Conj_H9_NotExtremal}
we know that the only extremal tripartite graphs we have to consider
are $H_7$ and $H_7'$. Let us show that $H_7'$ can do no better than
$H_7$.

Let $(\abg)\in R_2$ and $(H_7',w')\in\Tri(\abg)$. We need to show
there exists a weighting $w$ for $H_7$ so that $(H_7,w)$ has the
same densities as $(H_7',w')$. Note that
$\Gamma_A(b_2)=\Gamma_A(b_3)$ in $H_7'$ and $w'(b_2)+w'(b_3)>0$
(otherwise $\alpha=1$ which can not occur according to Lemma
\ref{abgInR3_implies_Not01} (ii)). Hence we can modify $H_7'$ by
applying $\Merge$ on $b_2, b_3$ labelling the resulting merged
vertex $b$. This creates one partial edge $bc_2$. Apply $\Split$ to
this edge to remove it, choosing to replace the vertex $c_2$. The
resulting weighted tripartite graph has the same densities as
$(H_7',w')$ and it is easy to check that it is strongly-isomorphic
to $H_7$.

Therefore when $(\abg)\in R_2$ we need only consider graphs
strongly-isomorphic to $H_7$, and by Lemma \ref{7Vertex_MinDensity}
we get $T_{\min}(\abg)$ is equal to
\begin{multline*}
\min\{2\sqrt{\alpha\beta(1-\gamma)}+2\gamma-2,
2\sqrt{\alpha\gamma(1-\beta)}+2\beta-2,
2\sqrt{\beta\gamma(1-\alpha)}+2\alpha-2\}.
\end{multline*}
To finish the proof let us show that $\gamma\leq\beta$ if and only
if
\[
2\sqrt{\alpha\gamma(1-\beta)}+2\beta-2\geq
2\sqrt{\alpha\beta(1-\gamma)}+2\gamma-2.
\]
We can prove a similar result for $\gamma\leq\alpha$. For ease of
notation let $d_1 = 2\sqrt{\alpha\gamma(1-\beta)}+2\beta-2$ and $d_2
= \alpha+\beta+\gamma-2$. So we have
\begin{center}
\begin{tabular}{lrcl}
& $d_1$ & $\geq$ & $2\sqrt{\alpha\beta(1-\gamma)}+2\gamma-2$\\
$\iff$ & $d_1+2(1-\gamma)$ & $\geq$ & $2\sqrt{\alpha\beta(1-\gamma)}$\\
$\iff$ & $(d_1+2(1-\gamma))^2$ & $\geq$ & $4\alpha\beta(1-\gamma)$\\
& & $=$ & $(d_2+2(1-\gamma))^2-\Delta(\abg)$\\
$\iff$ & $d_1^2+4d_1(1-\gamma)$ & $\geq$ & $d_2^2+4d_2(1-\gamma)-\Delta(\abg)$\\
$\iff$ & $d_1^2+4d_1-d_2^2-4d_2+\Delta(\abg)$ & $\geq$ &
$4\gamma(d_1-d_2)$
\end{tabular}
\end{center}
By Lemma \ref{a+b+g-2_LowerBound} we know $d_1-d_2\geq 0$. It is
easy to check that $d_1-d_2=0$ implies $\Delta(\abg)\geq 0$ which is
not true, since $(\abg)\in R_2$. Consequently we have
\[
d_1\geq 2\sqrt{\alpha\beta(1-\gamma)}+2\gamma-2 \iff
\frac{d_1^2+4d_1-d_2^2-4d_2+\Delta(\abg)}{4(d_1-d_2)}\geq\gamma.
\]
Substituting $d_1=2\sqrt{\alpha\gamma(1-\beta)}+2\beta-2$ and
$d_2=\alpha+\beta+\gamma-2$ into
\[
\frac{d_1^2+4d_1-d_2^2-4d_2+\Delta(\abg)}{4(d_1-d_2)}
\]
shows that it simplifies to $\beta$. Thus
\[
2\sqrt{\alpha\gamma(1-\beta)}+2\beta-2\geq
2\sqrt{\alpha\beta(1-\gamma)}+2\gamma-2 \iff \beta\geq\gamma.
\]
\end{proof}

\newpage

\section*{Appendix}

The following is a C++ implementation of the algorithm that produces
a list of possible extremal vertex minimal tripartite graphs. The
graphs it outputs are given in Figure \ref{Fig_Program_Output}.

\pagestyle{empty}

\begin{scriptsize}
\begin{verbatim}
#include <iostream>
#include <iomanip>
#include <fstream>
using namespace std;

//If bSaveToTxtFile is true it will save the output of the program to the file "Output.txt".
bool bSaveToTxtFile = false;

//CGraph stores data about a tripartite graph (with at most 3 vertices per vertex class, see
//Lemma 4.11).
//
//Member variables
//----------------
//
//M[][]       is the adjacency matrix of the graph.
//            An entry of 1 means the edge is present, 0 indicates it is missing.
//
//hasVertex[] tells us which vertices are part of the graph.
//            This allows us to represent smaller class sizes.
//            An entry of 1 means the vertex is part of the graph, 0 indicates it is missing.
//
//idEdge      is a 27 bit integer representation of M[][].
//            We use EdgeMatrix[][] to convert between M[][] and idEdge.
//
//idVertex    is a 9 bit integer representation of hasVertex[].
//            The least significant bit should correspond to hasVertex[0].
//
//classSize[] keeps track of the size of the vertex classes (for speed).
//            E.g. classSize[0] = hasVertex[0]+hasVertex[1]+hasVertex[2].
//
//Member functions
//----------------
//
//void AssignMatrixFromID(long id)   id should be a 27 bit integer. The function sets idEdge
//                                   to id, then uses EdgeMatrix[][] to construct the
//                                   corresponding adjacency matrix M[][].
//
//void AssignVerticesFromID(long id) id should be a 9 bit integer. The function sets idVertex
//                                   to id and then constructs the corresponding hasVertex[]
//                                   array. It also fills in classSize[].
//
//Notes
//-----
//
//The vertex classes are:
//Class 0 - vertices 0, 1, 2,
//Class 1 - vertices 3, 4, 5,
//Class 2 - vertices 6, 7, 8.
//
//Observe that vertex v lies in vertex class v/3 (where '/' is the integer division
//operator). This provides an efficient way of testing if two vertices lie in the same vertex
//class, which we will need to do often.
//
//If hasVertex[v] is zero it means vertex v should not be considered part of the tripartite
//graph. This is to allow CGraph to be flexible enough to represent tripartite graphs of
//smaller order. The adjacency matrix M[][] is of fixed size but since vertex v is not part
//of the graph, the entries in row v and column v are of no importance. Some functions assume
//(for speed purposes) that if the vertex is missing then its corresponding column and row
//entries should all be zero.
//
//To speed things up, functions that take CGraph objects as input do little (if any) checking
//that the objects are well-formed, e.g. entries of M[][] are all 0 or 1, M[][] is symmetric,
//that idEdge correctly represents M[][], or vertices that are not part of the graph have no
//neighbours. It is the responsibility of the process calling such functions to ensure that
//the CGraph objects are well-formed.
//
//We keep a record of the variables idEdge and idVertex as they provide a natural ordering of
//CGraph objects. This helps make the process of testing if two tripartite graphs are
//strongly-isomorphic to each other more efficient.
class CGraph
{
public:
    //Construct M[][] using EdgeMatrix[][] and id (a 27 bit integer), idEdge is set to id.
    void AssignMatrixFromID(long id);

    //Sets idVertex to id (a 9 bit integer), then constructs hasVertex[] and classSize[].
    void AssignVerticesFromID(long id);

    //The adjacency matrix of the graph.
    //0 = Edge missing.
    //1 = Edge present.
    long M[9][9];

    //Used to indicate which vertices are part of the graph.
    //0 = Vertex has been removed.
    //1 = Vertex is part of the graph.
    long hasVertex[9];

    //A 27 bit integer representing M[][].
    long idEdge;

    //A 9 bit integer representing hasVertex[].
    long idVertex;

    //The size of the vertex classes.
    long classSize[3];
};

//EdgeMatrix[][] is used to bijectively map the 27 possible edges of the tripartite graphs to
//27 bit integers. The entries indicate which bit (0 to 26) in CGraph::idEdge is the
//corresponding entry in the adjacency matrix CGraph::M[][]. Entries of -1 indicate that such
//an edge lies within a vertex class, so should always be zero in CGraph::M[][], and
//consequently does not get mapped to any of the 27 bits in CGraph::idEdge.
long EdgeMatrix[9][9] = {
    {-1, -1, -1, 18, 19, 20,  0,  1,  2},
    {-1, -1, -1, 21, 22, 23,  3,  4,  5},
    {-1, -1, -1, 24, 25, 26,  6,  7,  8},
    {18, 21, 24, -1, -1, -1,  9, 10, 11},
    {19, 22, 25, -1, -1, -1, 12, 13, 14},
    {20, 23, 26, -1, -1, -1, 15, 16, 17},
    { 0,  3,  6,  9, 12, 15, -1, -1, -1},
    { 1,  4,  7, 10, 13, 16, -1, -1, -1},
    { 2,  5,  8, 11, 14, 17, -1, -1, -1}};

//The 1296 permutations of the the 9 vertices that leaves the tripartite graph
//strongly-isomorphic to its original form.
long StrongIsoPermute[1296][9];

//Store the possible extremal vertex minimal graphs in Store[][0]. Also keep a copy of all
//graphs strongly-isomorphic to Store[i][0] in Store[i][1 to 1295]. Use Found to track how
//much of Store[][] has been filled in.
CGraph Store[200][1296];
long   Found;

//Graphs that are strongly-isomorphic to F7 and F9 (see Figure 8).
CGraph F7, F9;

//Displays the adjacency matrix of graph g. To help distinguish the vertex classes it uses
//the symbol '.' to represent edges that lie within a class.
void DisplayGraph(CGraph& g, ostream& os = cout);

//Display all the possible extremal vertex minimal graphs in Store[][0].
//Ignores graphs with idEdge = -1.
void DisplayStore(ostream& os = cout);

//Declared in CGraph. Construct M[][] using EdgeMatrix[][] and id (a 27 bit integer), idEdge
//is set to id.
//void CGraph::AssignMatrixFromID(long id);

//Declared in CGraph. Sets idVertex to id (a 9 bit integer), then constructs hasVertex[]
//and classSize[].
//void CGraph::AssignVerticesFromID(long id);

//Initializes the StrongIsoPermute[][] array and constructs the graphs F7 and F9.
void SetUp();

//Returns true if it can show the graph is not extremal or not vertex minimal (using Lemmas
//4.10, 4.13, and 4.14).
bool bBadClassSize(CGraph& g);

//Returns true if it can show the graph is not extremal or not vertex minimal (using
//Corollary 4.5).
bool bCanMoveEdgeWeights(CGraph& g);

//Returns true if g contains no triangles. We are not interested in such a graph by Theorem
//2.2.
bool bHasNoTriangles(CGraph& g);

//Returns true if it can show the graph is not extremal or not vertex minimal (using Lemma
//4.15). This function assumes Store[][] is filled with a complete list of possible extremal
//vertex minimal tripartite graphs. 
bool bCanReplaceBy8(CGraph& g);

//Creates a graph strongly-isomorphic to gIn by using the vertex mapping given in
//StrongIsoPermute[p][]. The result is stored in gOut (which should be a different object
//to gIn).
void StrongIsoGraph(CGraph& gIn, long p, CGraph& gOut);

//Displays the adjacency matrix of graph g. To help distinguish the vertex classes it uses
//the symbol '.' to represent edges that lie within a class.
void DisplayGraph(CGraph& g, ostream& os)
{
    long i,j;

    for(i=0;i<9;i++)
    {
        if(g.hasVertex[i]==0)
            continue; //Ignore vertices that are not part of the graph.

        for(j=0;j<9;j++)
        {
            if(g.hasVertex[j]==0)
                continue; //Ignore vertices that are not part of the graph.

            //Display whether ij is an edge.
            if(i/3==j/3)
                os << " ."; //The edge lies within a vertex class.
            else
            {
                if(g.M[i][j]==0)
                    os << " 0"; //The edge is missing.
                else
                    os << " 1"; //The edge is present.
            }
        }

        os << endl;
    }

    os << endl;

    return;
}

//Display all the possible extremal vertex minimal graphs in Store[][0].
//Ignores graphs with idEdge = -1.
void DisplayStore(ostream& os)
{
    long i,order;

    os << "Adjacency matrices of possible extremal vertex minimal tripartite" << endl;
    os << "graphs (edges which would lie within a vertex class are indicated" << endl;
    os << "by the entry \".\"):" << endl;
    os << endl;

    for(order=0;order<=9;order++) //Display the smallest order graphs first.
    for(i=0;i<Found;i++)
    {
        if(Store[i][0].idEdge==-1)
            continue; //The graph was removed.

        if(Store[i][0].classSize[0]+Store[i][0].classSize[1]+Store[i][0].classSize[2]!=order)
            continue; //The graph does not have the right number of vertices. 

        DisplayGraph(Store[i][0],os);
        
        os << endl;
    }

    os << endl;
}

//Construct M[][] using EdgeMatrix[][] and id (a 27 bit integer), idEdge is set to id.
void CGraph::AssignMatrixFromID(long id)
{
    long i,j;

    idEdge = id;

    //Fill in the adjacency matrix.
    for(i=0;i<9;i++)
    for(j=0;j<9;j++)
        if(i/3==j/3)
            M[i][j] = 0; //Edge lies within a class.
        else
            M[i][j] = (id>>EdgeMatrix[i][j])&1; //Assign the the correct bit of id.

    return;
}

//Sets idVertex to id (a 9 bit integer), then constructs hasVertex[] and classSize[].
void CGraph::AssignVerticesFromID(long id)
{
    long i;

    idVertex = id;
        
    for(i=0;i<9;i++)
        hasVertex[i] = (id>>i)&1; //Assign the "i"th bit of id.

    for(i=0;i<3;i++)
        classSize[i] = hasVertex[3*i]+hasVertex[3*i+1]+hasVertex[3*i+2];
    
    return;
}

//Initializes the StrongIsoPermute[][] array and constructs the graphs F7 and F9.
void SetUp()
{
    long i,n;
    long idEdge,idVertex;
    long P[4];

    //Fill in the StrongIsoPermute[][] array.
    {
        //Permute[][] stores the 6 perumtations of {0,1,2}.
        long Permute[6][3] = {{0,1,2}, {0,2,1}, {1,0,2}, {1,2,0}, {2,0,1}, {2,1,0}}; 

        //Permute class 0 by Permute[P[0]][].
        //Permute class 1 by Permute[P[1]][].
        //Permute class 2 by Permute[P[2]][].
        //Permute classes by Permute[P[3]][].
        n = 0;
        for(P[0]=0;P[0]<6;P[0]++)
        for(P[1]=0;P[1]<6;P[1]++)
        for(P[2]=0;P[2]<6;P[2]++)
        for(P[3]=0;P[3]<6;P[3]++)
        {
            //Calculate the new label for vertex i.
            //i is the "i%3" vertex in class "i/3", i.e. i = 3*(i/3)+(i%3).
            //Class "i/3" gets mapped to class Permute[P[3]][i/3].
            //The vertices in class "i/3" are permuted by Permute[P[i/3]][].
            for(i=0;i<9;i++)
                StrongIsoPermute[n][i] = 3*Permute[P[3]][i/3]+Permute[P[i/3]][i%3];

            n++;
        }
    }

    //Construct the graph F7 (see Figure 8).
    {
        //Vertices 7,6,5,4,3,1,0 are part of F7.
        //Vertices 8 and 2 are not part of F7.
        idVertex = 251; //Binary 011111011.
        
        //An array of the 9 edges in F7.
        long Edge[9][2] = {{0,4}, {0,5}, {0,7}, {1,3}, {1,4}, {1,6}, {3,7}, {4,6}, {5,6}};
        
        //Construct idEdge for F7 using Edge[][] and EdgeMatrix[][].
        idEdge = 0;
        for(i=0;i<9;i++)
            idEdge |= 1<<EdgeMatrix[Edge[i][0]][Edge[i][1]];

        F7.AssignVerticesFromID(idVertex);
        F7.AssignMatrixFromID(idEdge);
    }

    //Construct the graph F9 (see Figure 8).
    {
        //Vertices 0 to 8 are part of F9.
        idVertex = 511; //Binary 111111111.
        
        //An array of the 15 edges in F9.
        long Edge[15][2] = {
            {0,4}, {0,5}, {0,7}, {1,3}, {1,5},
            {1,6}, {1,8}, {2,4}, {2,6}, {2,7},
            {3,7}, {3,8}, {4,6}, {4,8}, {5,7}};
        
        //Construct idEdge for F9 using Edge[][] and EdgeMatrix[][].
        idEdge = 0;
        for(i=0;i<15;i++)
            idEdge |= 1<<EdgeMatrix[Edge[i][0]][Edge[i][1]];

        F9.AssignVerticesFromID(idVertex);
        F9.AssignMatrixFromID(idEdge);
    }

    return;
}

//Returns true if it can show the graph is not extremal or not vertex minimal (using Lemmas
//4.10, 4.13, and 4.14).
bool bBadClassSize(CGraph& g)
{
    //Our tests will involve testing the neighbourhoods of vertex 0 and 1.
    if(g.hasVertex[0]==0 || g.hasVertex[1]==0)
        return false;

    //Check whether vertices 0 and 1 have the same neighbours in class 1.
    if(g.M[0][3]==g.M[1][3] && g.M[0][4]==g.M[1][4] && g.M[0][5]==g.M[1][5])
    {
        if(g.classSize[2]==3)
            return true; //g is not extremal or not vertex minimal by Lemma 4.13.

        if(g.classSize[0]==2)
            return true; //g is not extremal by Lemma 4.10.
        
        if(g.M[0][3]==g.M[2][3] && g.M[0][4]==g.M[2][4] && g.M[0][5]==g.M[2][5])
        {
            //Vertices 0, 1, and 2 have the same neighbours in class 1, hence by Lemma 4.10
            //g is not extremal.
            return true;
        }

        if(g.M[0][6]==g.M[1][6] && g.M[0][7]==g.M[1][7] && g.M[0][8]==g.M[1][8])
        {
            //Vertices 0 and 1 have the same neighbourhood, hence by Lemma 4.14 g is not
            //vertex minimal.
            return true;
        }
    }
    
    return false;
}

//Returns true if it can show the graph is not extremal or not vertex minimal (using
//Corollary 4.5).
bool bCanMoveEdgeWeights(CGraph& g)
{
    long i,j;
    bool bC0,bC1;
    bool bIsEqual, bIsSubset;

    //We will try to apply Corollary 4.5 to
    //(i)  the edge 0,3 and the missing edge 0,4
    //(ii) the edge 0,3 and the missing edge 1,4

    if(g.hasVertex[0]==0 || g.hasVertex[1]==0
    || g.hasVertex[3]==0 || g.hasVertex[4]==0)
        return false; //A vertex is missing so cannot apply the test.

    if(g.M[0][3]==0)
        return false; //Edge 0,3 is missing.

    for(j=0;j<=1;j++)
        if(g.M[j][4]==0)
        {
            //g contains edge 0,3 and j,4 is missing (where j = 0 or 1).

            //Check C_{j,4} is a proper subset of C_{0,3}
            bIsEqual  = true;
            bIsSubset = true;
            for(i=6;i<9;i++)
            {
                bC0 = (g.M[i][j]==1 && g.M[i][4]==1);
                bC1 = (g.M[i][0]==1 && g.M[i][3]==1);

                if(bC0==true && bC1==false)
                    bIsSubset = false;

                if(bC0!=bC1)
                    bIsEqual = false;
            }

            if(bIsSubset==true && bIsEqual==false)
                return true; //g is not extremal or not vertex minimal by Corollary 4.5.
        }

    return false;
}

//Returns true if g contains no triangles. We are not interested in such a graph by Theorem
//2.2.
bool bHasNoTriangles(CGraph& g)
{
    long i,j,k; //Vertices in classes 0,1,2 respectively.

    //Note that if a vertex is missing it has no neighbours so cannot be part of a triangle.
    //This saves us having to use the hasVertex[] array.
    for(i=0;i<3;i++)
    for(j=3;j<6;j++)
    for(k=6;k<9;k++)
        if(g.M[i][j]==1 && g.M[i][k]==1 && g.M[j][k]==1)
            return false; //ijk forms a triangle.

    return true; //Could not find a triangle.
}

//Returns true if it can show the graph is not extremal or not vertex minimal (using Lemma
//4.15). This function assumes Store[][] is filled with a complete list of possible extremal
//vertex minimal tripartite graphs. 
bool bCanReplaceBy8(CGraph& g)
{
    long i,j,k,n;
    long a0,b0,a1,b1;
    long G1[10][10];
    long G2[10][10];
    bool bC0,bC1;
    bool bSame;

    //ReplaceEdgeID[] Holds the CGraph::idEdge values for the eight graphs described in Lemma
    //4.15.
    long ReplaceEdgeID[8];
    
    //We will try to apply Lemma 4.15 to
    //(i)   the edge 6,0 and the missing edge 6,1,
    //(ii)  the edge 6,0 and the missing edge 7,0,
    //(iii) the edge 6,0 and the missing edge 7,1.
    //Where class A in Lemma 4.15 is vertex class 2.
    long cases[3][2] = {{6,1},{7,0},{7,1}};
    
    //a1 = 6,           b1 = 0,
    //a0 = cases[n][0], b0 = cases[n][1].

    if(g.classSize[2]!=3)
        return false; //The class in which vertices a0, a1 lie in must be of size 3.

    if(g.hasVertex[0]==0 || g.hasVertex[6]==0
    || g.hasVertex[1]==0 || g.hasVertex[7]==0)
        return false; //A vertex is missing so cannot apply the test.

    if(g.M[0][6]==0)
        return false; //Edge 6,0 is missing.

    a1 = 6;
    b1 = 0;
    for(n=0;n<3;n++)
    {
        //Find suitable values for a0,b0, a1,b1.
        {
            a0 = cases[n][0];
            b0 = cases[n][1];

            if(g.M[a0][b0]==1)
                continue;

            //Check C_{a0,b0} = C_{a1,b1}
            bSame = true;
            for(i=3;i<6;i++)
            {
                bC0 = (g.M[i][a0]==1 && g.M[i][b0]==1);
                bC1 = (g.M[i][a1]==1 && g.M[i][b1]==1);

                if(bC0!=bC1)
                    bSame = false;
            }

            if(bSame==false)
                continue;
        }

        //Create graphs G1, and G2.
        {
            //Copy g into G1 and G2.
            for(i=0;i<9;i++)
            for(j=0;j<9;j++)
            {
                G1[i][j] = g.M[i][j];
                G2[i][j] = g.M[i][j];
            }

            //Remove edge a1,b1 from G1.
            G1[a1][b1] = 0;
            G1[b1][a1] = 0;

            //Add edge a0,b0 to G2.
            G2[a0][b0] = 1;
            G2[b0][a0] = 1;

            //Fill in the neighbours for the vertex a2 in column/row 9.
            for(i=0;i<9;i++)
            {
                //Copy vertex a0.
                G1[9][i] = G1[a0][i];
                G1[i][9] = G1[i][a0];

                //Copy vertex a1.
                G2[9][i] = G2[a1][i];
                G2[i][9] = G2[i][a1];
            }

            //Fill in the entry a2,a2.
            G1[9][9] = 0;
            G2[9][9] = 0;

            //Add edge a2,b0 to G1.
            G1[9][b0] = 1;
            G1[b0][9] = 1;

            //Remove edge a2,b1 from G2.
            G2[9][b1] = 0;
            G2[b1][9] = 0;
        }

        //Fill the array ReplaceEdgeID[] by calculating the idEdge of G1, G2 once a vertex
        //has been removed from class 2.
        {
            //map[][] is used to remove vertex k+6 in class 2 by mapping vertex i to
            //vertex map[k][i] in G1, G2;
            long map[4][9] = {
                {0,1,2,3,4,5,  7,8,9},
                {0,1,2,3,4,5,6,  8,9},
                {0,1,2,3,4,5,6,7,  9},
                {0,1,2,3,4,5,6,7,8  }};

            for(k=0;k<4;k++)
            {
                ReplaceEdgeID[k] = 0;
                for(i=0;i<9;i++)
                for(j=0;j<9;j++)
                    if(i/3!=j/3)
                        ReplaceEdgeID[k] |= G1[map[k][i]][map[k][j]]<<EdgeMatrix[i][j];
            }

            for(k=0;k<4;k++)
            {
                ReplaceEdgeID[k+4] = 0;
                for(i=0;i<9;i++)
                for(j=0;j<9;j++)
                    if(i/3!=j/3)
                        ReplaceEdgeID[k+4] |= G2[map[k][i]][map[k][j]]<<EdgeMatrix[i][j];
            }
        }

        //Check if any of the possible extremal vertex minimal graphs have the same idEdge as
        //ReplaceEdgeID[] by searching through Store[][].
        {
            bSame = false;
            for(i=0;i<8;i++)
            for(j=0;j<Found;j++)
            {
                if(bSame==true)
                {
                    //A graph in Store[][] has an idEdge which matches a member of
                    //ReplaceEdgeID[].
                    break;
                }

                if(Store[j][0].idEdge==-1)
                {
                    //The graph Store[j][0] is not an extremal vertex minimal graph, and
                    //hence neither are any of the graphs that are strongly-isomorphic to it.
                    continue;
                }

                for(k=0;k<1296;k++)
                    if(Store[j][k].idEdge==ReplaceEdgeID[i])
                    {
                        //One of the eight subgraphs of G1, G2 is possibly extremal and
                        //vertex minimal. So g maybe extremal and vertex minimal.
                        bSame = true;
                        break;
                    }
            }

            if(bSame==false)
            {
                //The eight subgraphs of G1, G2 are not extremal or not vertex minimal, hence
                //by Lemma 4.15 the graph g is not extremal or not vertex minimal.
                return true;
            }
        }
    }

    return false;
}

//Creates a graph strongly-isomorphic to gIn by using the vertex mapping given in
//StrongIsoPermute[p][]. The result is stored in gOut (which should be a different object
//to gIn).
void StrongIsoGraph(CGraph& gIn, long p, CGraph& gOut)
{
    long i,j;

    //Map vertex i in gIn to vertex StrongIsoPermute[p][i] in gOut.

    //Fill in hasVertex[].
    for(i=0;i<9;i++)
        gOut.hasVertex[StrongIsoPermute[p][i]] = gIn.hasVertex[i];

    //Fill in M[][].
    for(i=0;i<9;i++)
    for(j=0;j<9;j++)
        gOut.M[StrongIsoPermute[p][i]][StrongIsoPermute[p][j]] = gIn.M[i][j];

    //Create the idVertex.
    gOut.idVertex = 0;
    for(i=0;i<9;i++)
        gOut.idVertex |= gOut.hasVertex[i]<<i;

    //Create the idEdge.
    gOut.idEdge = 0;
    for(i=0;i<9;i++)
    for(j=0;j<9;j++)
        if(i/3!=j/3)
            gOut.idEdge |= gOut.M[i][j]<<EdgeMatrix[i][j];

    //Fill in classSize[].
    for(i=0;i<3;i++)
        gOut.classSize[i] = gOut.hasVertex[3*i]+gOut.hasVertex[3*i+1]+gOut.hasVertex[3*i+2];

    return;
}

int main()
{
    long   i,j,n,p,v;
    long   idVertex[64];
    long   notIsolatedID;
    long   idEdge;
    bool   bStoreChanged;
    CGraph g;

    //Initialize StrongIsoPermute[][] and construct F7, F9.
    SetUp();

    //Go through all possible CGraph::idVertex values and store only those which have class
    //sizes that are two or three. We are not interested in graphs with empty vertex classes
    //by Theorem 2.2, or vertex classes of size one by Lemma 4.8.
    n = 0;
    for(i=0;i<(1<<9);i++)
    {
        g.AssignVerticesFromID(i);
        
        if(g.classSize[0]>1 && g.classSize[1]>1 && g.classSize[2]>1)
        {
            idVertex[n] = i;
            n++;
        }
    }

    //Initialize the number of graphs found that are possibly extremal and vertex minimal.
    Found = 0;

    //Generate all possible tripartite graphs by cycling through all possible 27 bit
    //integers representing CGraph::idEdge, and all possible CGraph::idVertex values stored
    //in idVertex[].
    for(idEdge=0;idEdge<(1<<27);idEdge++)
    {
        g.AssignMatrixFromID(idEdge);

        //Work out which vertices have neighbours, and store the result in notIsolatedID.
        //This will aid us in creating a well-formed CGraph object. Much like
        //CGraph::idVertex, if vertex i has a neighbour then bit i is 1 otherwise it is 0.
        //Bit 0 is the least significant bit.
        notIsolatedID = 0;
        for(i=0;i<9;i++)
        for(j=0;j<9;j++)
            notIsolatedID |= g.M[i][j]<<i;

        for(v=0;v<64;v++)
        {
            if((idVertex[v] & notIsolatedID)!=notIsolatedID)
            {
                //If we assign idVertex[v] to g then g will not be well-formed as there will
                //exist a non-isolated vertex according to M[][] which is not part of the
                //graph according to hasVertex[].
                continue;
            }
            
            g.AssignVerticesFromID(idVertex[v]);
            
            //Check if we have enough memory to store the graph in Store[][]. Using the
            //sizeof operator is a standard trick to get the size of an array.
            //sizeof(Store)/sizeof(Store[0]) should equal 200.
            if(Found>=sizeof(Store)/sizeof(Store[0]))
            {
                //We should never run out of memory.
                cout << endl << endl;
                cout << "Error: Too many graphs, increase size of Store[][]." << endl;
                cout << endl; 
                return 0;
            }

            //Start filling in Store[Found][] with all the graphs strongly-isomorphic to g. 
            for(p=0;p<1296;p++)
            {
                StrongIsoGraph(g,p,Store[Found][p]);

                //If for any reason we wish to discard g and all its isomorphisms, we will
                //break out of this loop and deal with removing the graphs from Store[][].

                //Note that if Store[x][0] and Store[y][0] are strongly-isomorphic to each
                //other. Then the arrays Store[x][] and Store[y][] are just permutations of
                //each other. To avoid such duplication we'll only keep Store[Found][] if the
                //idEdge and idVertex of Store[Found][0] is the smallest of all the
                //isomorphisms in Store[Found][].            
                if( Store[Found][0].idEdge   > Store[Found][p].idEdge
                || (Store[Found][0].idEdge  == Store[Found][p].idEdge
                &&  Store[Found][0].idVertex > Store[Found][p].idVertex))
                    break;

                //Get rid of graphs whose class size can be reduced, or is too small (i.e.
                //those satisfying Lemmas 4.10, 4.13, and 4.14).
                if(bBadClassSize(Store[Found][p])==true)
                    break;

                //Get rid of graphs where we can reduce the density of triangles by moving
                //edge weights (see Corollary 4.5).
                if(bCanMoveEdgeWeights(Store[Found][p])==true)
                    break;

                //Get rid of graphs without a triangle (see Theorem 2.2).
                if(bHasNoTriangles(Store[Found][p])==true)
                    break;

                //Get rid of graphs strongly-isomorphic to F7 (see Lemma 4.16).
                if(Store[Found][p].idEdge  ==F7.idEdge
                && Store[Found][p].idVertex==F7.idVertex)
                    break;

                //Get rid of graphs strongly-isomorphic to F9 (see Lemma 4.18).
                if(Store[Found][p].idEdge  ==F9.idEdge
                && Store[Found][p].idVertex==F9.idVertex)
                    break;
            }

            //If p!=1296 we broke out of the loop earlier than expected because
            //   (i)  one of graphs was not extremal or not vertex minimal,
            //or (ii) Store[Found][0] was not the graph with the smallest IDs.
            //In either case we wish to discard everything in Store[Found][] which we can
            //easily do by not updating Found.

            if(p==1296)
            {
                //We didn't break out of the loop early. We wish to retain the graphs in
                //Store[Found][] which we do by increasing the value of Found. 
                Found++;
            }
        }
    }

    //We have now tested every possible tripartite graph (with class sizes at most three).
    //Store[][] holds a list of possible extremal vertex minimal graphs, which we will reduce
    //further by repeated applications of Lemma 4.15, as implemented by the function
    //bCanReplaceBy8(). If Store[i][j] satisfies bCanReplaceBy8() then Store[i][j] and all
    //graphs strongly-isomorphic to it (i.e. the graphs in the Store[i][] array) can be
    //removed. For speed and convenience we will indicate that they have been removed by
    //setting Store[i][0].idEdge to -1.

    do
    {
        bStoreChanged = false;

        //Go through each possible extremal vertex minimal graph.
        for(i=0;i<Found;i++)
        {
            if(Store[i][0].idEdge==-1)
                continue; //The graph and its isomorphisms have been removed.

            for(j=0;j<1296;j++)
                if(bCanReplaceBy8(Store[i][j])==true)
                {
                    Store[i][0].idEdge = -1;

                    //bCanReplaceBy8() depends on Store[][]. Since Store[][] has changed
                    //there may be graphs that returned false but now will return true. Hence
                    //we need to re-check those graphs and repeat this process, which we
                    //indicate by setting bStoreChanged to true.

                    bStoreChanged = true;
                    break;
                }
        }
    }while(bStoreChanged==true);

    //Display the final list of possible extremal vertex minimal graphs.
    DisplayStore();

    //Save results to "Output.txt".
    if(bSaveToTxtFile==true)
    {
        fstream OutputFile;
        OutputFile.open("Output.txt", fstream::out | fstream::trunc);
        if(OutputFile.fail())
        {
            cout << endl << endl;
            cout << "Failed to open Output.txt." << endl;
            cout << endl;
            return 0;
        }
        DisplayStore(OutputFile);
        OutputFile.close();
    }

    cout << "Finished." << endl << endl;
    return 0;
}
\end{verbatim}
\end{scriptsize}


\begin{thebibliography}{9}
\bibitem{Boll} B. Bollob\'as, \emph{Relations between sets of complete subgraphs},
in Proceedings of the Fifth British Combinatorial Conference (Eds C.
St.J. A. Nash-Williams and J. Sheehan), Congr. Numer., \textbf{15}
(1976), 79--84.
\bibitem{Paper_BSTT} A. Bondy, J. Shen, S. Thomass\'e, C. Thomassen,
\emph{Density conditions for triangles in multipartite graphs},
Combinatorica, \textbf{26} (2) (2006), 121--131.
\bibitem{Erd62} P. Erd\H os, \emph{On a theorem of Rademacher--Tur\'an}, Illinois J. Math. \textbf{6} (1962), 122--127.
\bibitem{Fisher} D.C. Fisher, \emph{Lower bounds on the number of triangles in a graph},
J. Graph Theory, \textbf{13} (4) (1989), 505--512.
\bibitem{LS} L. Lov\'asz and M. Simonovits, \emph{On the number of complete subgraphs of a graph, II},
Studies in pure mathematics, Birkh\"auser, (1983), 459--495.
\bibitem{Man} V. W. Mantel, \emph{Problem 28}, Wiskundige Opgaven 10, (1907), 60--61.
\bibitem{Raz} A. A. Razborov \emph{On the minimal density of triangles in graphs},
Combin. Probab. and Comput., \textbf{17} (4), (2008), 603--618.
\end{thebibliography}
\end{document}